\documentclass[oneside,reqno]{amsart}
\usepackage{amssymb}
\usepackage{amsfonts}
\usepackage{amsmath}
\usepackage{graphicx}
\providecommand{\U}[1]{\protect \rule{.1in}{.1in}}
\newtheorem{theorem}{Theorem}
\theoremstyle{plain}

\newtheorem{proposition}{Proposition}
\newtheorem{remark}{Remark}

\numberwithin{equation}{section}
\usepackage[utf8]{inputenc}

\title [Neural Network Operators of Convolution Type]{ Approximation by  Neural Network  Operators of Convolution Type Activated by Deformed and Parametrized Half Hyperbolic Tangent Function }
\begin{document}
\author{Asiye Arif$^{1}$}
\author{Tuğba Yurdakadim$^{1}$}

\maketitle
\noindent$^{1}$Bilecik Şeyh Edebali University, Department of Mathematics, Turkey\\
E-mail: arif\_asiya@mail.ru$^1$, tugba.yurdakadim@bilecik.edu.tr $^1$,\\[2ex]

\begin{abstract}
Here, we introduce three kinds of neural network operators of convolution type which are activated by $q$-deformed and $\beta$-parametrized half hyperbolic tangent function. We obtain quantitative convergence results to the identity operator with the use of modulus of continuity. Global  smoothness preservation of our operators are also presented and the iterated versions of them are taken into the consideration. 
\end{abstract}

\section{Introduction and Motivation}
The main idea behind neural networks and artificial neural networks is to act like human brain by understanding the biological nature and function of it. Neural and artificial neural networks lay in the centre of machine learning and the studies on artificial intelligence that can learn and solve problems go back to   1950's \cite{rozenbaltt}. Recently, the power of computers, accessing the big data, studies on neural network, machine learning and deep learning gain speed and take noteworthy attention since they give us opportunity to handle complex, nonlinear  relations in big data, by making them suitable for the duties such as image recognition, natural language processing, optimization, process control system, forecasting, approximation of functions and solutions of curve fitting \cite{Arif}, \cite{Ismailaslan},  \cite{ALADAg},  \cite{ARMSTRONG}, \cite{amatod}, \cite{BISHOP}, \cite{Cardaliaguet}, \cite{CHUNG},  \cite{cIFTER},  \cite{FAALSIDE},   \cite{PARKER}.\\
\indent Quantitative approximation of positive linear operators to the unit operator has been studied by G. A. Anastassiou since 1985 \cite{anast 1985}, \cite{anast 1993}, \cite{anast 2001}, \cite{anast 2025}. By originating from the quantitative weak convergence of finite positive measures to the unit Dirac measure, having as a method the geometric moment theory \cite{anast 1993}, he has obtained best upper bounds and these studies have been considered from all possible perspectives, univariate and multivariate cases by many authors. In the present study, we introduce three kinds of convolution operators with the kernel depending on symmetrized, $q$-deformed and $\beta$-parametrized half hyperbolic tangent function. It is  noteworthy to mention that this activation function is frequently used in neural networks and they can be interpreted as positive linear operators, thus we use the methods of positive linear operator theory in our proofs. We present  quantitative convergence results  to the identity operator by using modulus of continuity  and global smoothness preservation of our operators are also obtained. Furthermore, iterated versions of these operators are taken into consideration. \\
\indent The outline of the paper is as follows: In  section 1, we explain the history of neural network theory  and our motivation. Section 2 is devoted to preliminaries. In  section 3, we present convergence results to the unit by our operators. In  section 4, the iterated versions of these operators are considered. We find it is valuable to mention that general motivation comes from seminal works \cite{G.A.2025b}, \cite{Bojanic}, \cite{Jung}, \cite{Moldovan2}, \cite{Moldovan}, \cite{Kamzolov}, \cite{Swetits}.

\section{ Preliminaries}
Let us recall the following deformed and parametrized half hyperbolic tangent function and its basic properties as below
\cite{Anastassiou20232}:
\begin{equation}
\nu_{q}(x)=\frac{1-qe^{-\beta x}}{1+qe^{-\beta  x}},~ x \in \mathbb{R},~ q,~\beta >0,
\end{equation}
\begin{itemize}
\item[•] $\nu_{q }(+\infty)=1, ~~\nu_{q,\beta }(-\infty)=-1,$
\item[•] $\nu_{q }(-x)=-\nu_{\frac{1}{q}}(x),~ for~every~ x \in (-\infty, \infty),$
\item[•] $ \nu_{q}(0)=\frac{1-q}{1+q}.$
\end{itemize}
Then with the use of $\nu$, the following activation function  has been considered  and its properties have been given in \cite{Anastassiou20232}:
\begin{equation}
G_{q,\beta }(x)=\frac{1}{4}\left(\nu_{q }(x+1)-\nu_{q }(x-1)\right),~x \in (-\infty, \infty).
\end{equation}
One can easily notice that 
\begin{align*}
G_{q, \beta }(-x)=&
\frac{1}{4}\left(\nu_{\frac{1}{q},\beta }(x+1)-\nu_{\frac{1}{q}, \beta }(x-1)\right) =G_{ \frac{1}{q}, \beta }(x), 
\end{align*}

\begin{equation}
\label{symmetry}
G_{q,\beta }(-x)=G_{ \frac{1}{q}, \beta }(x), 
\end{equation} 
which means that we have a deformed symmetry and the global maximum of $G_{q,\beta}(x)$ is
\begin{equation*}
G_{q,  \beta }\left( \frac{\ln q}{\beta} \right)= \frac{1-e^{-\beta }}{2(1+e^{ -\beta })}.
\end{equation*}
\noindent It is known from \cite{Anastassiou20232} that  $ G_{q, \beta }$ is a density function on $(-\infty, \infty)$ since
\begin{equation*}
\int_{-\infty}^{\infty} G_{q, \beta }(x)dx=1.
\end{equation*}

By observing the following symmetry 
\begin{equation*}
\left( G_{q,\beta } + G_{ \frac{1}{q}, \beta } \right)(-x)=\left( G_{q, \beta } + G_{ \frac{1}{q}, \beta } \right)(x),
\end{equation*}
we introduce 
\begin{equation*}
\Psi= \frac{ G_{q, \beta }+ G_{ \frac{1}{q}, \beta }}{2}.
\end{equation*}
Since our aim is to get a symmetrized density function, this definition of $\Psi$ serves to our aim, i.e.,
 \begin{equation*}
 \Psi(x)=\Psi(-x),~ x \in (-\infty, \infty)
 \end{equation*}
 and
\begin{equation*}
\int_{-\infty}^{ \infty} \Psi(x)dx=1, 
\end{equation*}
implying
\begin{equation}
\int_{-\infty}^{\infty} \Psi(nx-v)dv=1, for~every~ n \in \mathbb{N},~ x \in (-\infty, \infty).
\end{equation}

\section{Main Results}
Here, we introduce the following convolution type operators activated by symmetrized, $ q $-deformed and parametrized half hyperbolic tangent function. Our aim in this section is to examine their approximation properties and improve these results. \\
\indent Now by considering the above new density function $\Psi $, we introduce the following neural network operators of convolution type for $f\in C_{B}(\mathbb{R})$ consisting of the functions which are continuous and bounded on $\mathbb{R}$ and $n\in \mathbb{N}$ :
\begin{equation}
\label{3.1}
\mathtt{B}_n (f)(x):=\int_{-\infty}^{ \infty} f \left( \frac{v}{n}\right) \Psi(nx-v) dv,
\end{equation}
\begin{align}
\label{3.2}
\mathtt{B}_n^*(f)(x):&= n \int_{-\infty}^{ \infty} \left( \int^{\frac{v+1}{n}}_{\frac{v}{n}} f \left( h \right)dh\right)\Psi(nx-v) dv\\
\nonumber &=n \int_{-\infty}^{\infty} \left( \int^{\frac{1}{n}}_0 f \left( h+\frac{v}{n}\right)dh\right)\Psi(nx-v) dv,
\end{align}
\begin{equation}
\label{3.3}
 \overline{\mathtt{B}_n}(f)(x):=\int_{-\infty}^{ \infty} \left( \sum_{s=1}^r w_s f \left( \frac{v}{n}+\frac{s}{nr}\right)\right) \Psi(nx-v) dv,
\end{equation}
where $w_s \geq 0$, $\displaystyle \sum_{s=1}^r w_s=1$ and 
(\ref{3.2}) is  called  activated Kantorovich type of (\ref{3.1}),
(\ref{3.3}) is  called  activated  Quadrature type of (\ref{3.1}).\\
Here is our first result. 
\begin{theorem}
\label{t1}
\begin{equation}
\int_{\{v\in\mathbb{R}:|nx-v|\geq n^{1-\alpha}\}}\Psi(nx-v) dv< \frac{\left( q+\frac{1}{q} \right) }{e^{ \beta (n^{1-\alpha}-1)}}
\end{equation} holds for  $0<\alpha<1$, $n\in \mathbb{N}$ such that $n^{1-\alpha}>2$. 
\end{theorem}
\begin{proof}
Since 
\begin{equation}
G_{q, \beta }(x)=\frac{1}{4}\left(\nu_{q, \beta }(x+1)-\nu_{q, \beta }(x-1)\right),~x \in (-\infty, \infty),
\end{equation}
first letting 
$ x \in [1, \infty)$, and using Mean Value Theorem, we get 
\begin{equation}
G_{q, \beta }(x)=\frac{1}{4} \nu'_{q, \beta }(\xi)2=\frac{1}{2}\nu'_{q, \beta }(\xi)=\frac{q  \beta e^{\beta \xi}}{(e^{\beta \xi}+q  )^2  },~ 0\leq x-1<\xi< x+1.
\end{equation}
Also, we obtain  that 
\begin{equation*}
G_{q, \beta }(x)<\frac{q  \beta }{e^{\beta  \xi}}<\frac{q  \beta }{e^{\beta(x-1)}}
\end{equation*}
and similarly,  
\begin{equation*}
G_{ \frac{1}{q}, \beta }(x)<\frac{\frac{1}{q} \beta} { e^{\beta  \xi}} < \frac{\frac{1}{q} \beta} { e^{\beta  (x-1)}}
\end{equation*}
which imply 
\begin{align*}
\Psi<\frac{1}{2}\left( \frac{q  \beta }{e^{\beta(x-1)}}+  \frac{\frac{1}{q} \beta} { e^{\beta  (x-1)}}\right)\\
\Psi<\frac{1}{2}\left( q+\frac{1}{q} \right) \beta e^{-\beta(x-1) }.
\end{align*}

Let 
\begin{equation*}
F:=\{ v\in \mathbb{R}:|nx-v|\geq n^{1-\alpha} \}
\end{equation*}
and then we can write that 
\begin{align*}
\int_F \Psi(nx-v) dv&=\int_F \Psi(|nx-v|) dv<\frac{1}{2}\left( q+\frac{1}{q} \right) \beta \int_F e^{-\beta (|nx-v|-1)} dv\\
&=2 \cdot \frac{1}{2}\left( q+\frac{1}{q} \right) \beta  \int_{n^{1-\alpha}}^{\infty} e^{-\beta (x-1)}dx\\
&=\left( q+\frac{1}{q} \right) \beta  \int_{n^{1-\alpha}-1}^{\infty} e^{-\beta  y} dy \\
&=\left( q+\frac{1}{q} \right) e^{-\beta y} \vert_{\infty}^{n^{\alpha-1}-1}=\left( q+\frac{1}{q} \right) e^{-\beta (n^{1-\alpha}-1)}\\
&=\frac{\left( q+\frac{1}{q} \right)}{e^{\beta (n^{1-\alpha}-1)}}.
\end{align*}
Hence, we complete the proof. 
\end{proof}
Recall the modulus of continuity  of a function $f$ for $\theta >0$ on $\mathbb{R}$ as follows:
\begin{displaymath}
\omega(f,\theta):=\sup_{ \begin{subarray}{l}
x,y \in \mathbb{R}\\
|x-y|\leq\theta \end{subarray} } |f(x)-f(y)|.
\end{displaymath}
Now we are ready to present quantitative convergence results for our operators. Note that for Theorems 2, 3, 4, we let  $0<\alpha<1$, $n\in \mathbb{N}$ such that $n^{1-\alpha}>2$. 
\begin{theorem}
\label{theorem3}
\begin{equation*}
\left| \mathtt{B}_n (f)(x)-f(x)\right|\leq\omega\left(f,\frac{1}{n^{\alpha}}\right)+ \frac{2\left( q+\frac{1}{q} \right)  \Vert f \Vert_{\infty}}{e^{ \beta (n^{1-\alpha}-1)}}=:\mathtt{T}
\end{equation*}
and 
\begin{equation*}
\Vert \mathtt{B}_n(f)-f\Vert_{\infty} \leq \mathtt{T}
\end{equation*} hold for   $f\in C_B(\mathbb{R})$. 
Furthermore,  we have  that $ \displaystyle \lim_{n\rightarrow \infty} \mathtt{B}_n(f)=f$, pointwise and uniformly for $f\in C_{UB}(\mathbb{R})$, consisting of the functions which are uniformly continuous and bounded on $\mathbb{R}$.
\end{theorem} 
\begin{proof}
Let
$ \displaystyle F_1:=\left\lbrace v\in \mathbb{R}:\left| \frac{v}{n}-x \right|<\frac{1}{n^{\alpha}} \right \rbrace$,
$ \displaystyle F_2:=\left\lbrace v\in \mathbb{R}:\left| \frac{v}{n}-x \right| \geq \frac{1}{n^{\alpha}} \right \rbrace.$
Then $F_1\cup F_2=\mathbb{R}$ and we can write by using Theorem \ref{t1} that 
\begin{align*}
\vert \mathtt{B}_n(f)(x)-f(x)\vert&=\left\vert \int_{-\infty}^{ \infty} f\left(\frac{v}{n}\right) \Psi(nx-v)dv-f(x)\int_{-\infty}^{ \infty} \Psi(nx-v)dv \right\vert\\
&\leq  \int_{-\infty}^{ \infty} \left\vert f\left(\frac{v}{n}\right)-f(x)\right\vert \Psi(nx-v)dv\\
&= \int_{F_1} \left\vert f \left(\frac{v}{n}\right)-f(x)\right\vert \Psi(nx-v)dv+  \int_{F_2} \left\vert f \left(\frac{v}{n}\right)-f(x)\right\vert \Psi(nx-v)dv\\
&\leq \int_{F_1} \omega \left( f, \left\vert \frac{v}{n}-x \right\vert \right) \Psi(nx-v)dv+  2 \Vert f \Vert_{\infty}\int_{F_2}  \Psi(nx-v)dv\\
& \leq \omega \left( f, \frac{1}{n^{\alpha}} \right)+  \frac{  2 \Vert f \Vert_{\infty} \left( q+\frac{1}{q} \right)}{e^{\beta (n^{1-\alpha}-1)}}
\end{align*}
which gives the desired results. 
\end{proof}
\begin{theorem}
\label{theorem4}
\begin{equation*}
\left\vert \mathtt{B}_n^* (f)(x)-f(x)\right\vert \leq\omega\left(f,\frac{1}{n}+\frac{1}{n^{\alpha}}\right)+ \frac{2\left( q+\frac{1}{q} \right)\Vert f \Vert_{\infty}}{e^{\beta (n^{1-\alpha}-1)}}=:\mathtt{E}
\end{equation*} 
and 
\begin{equation*}
\Vert \mathtt{B}_n^*(f)-f\Vert_{\infty} \leq \mathtt{E}
\end{equation*} hold for $f\in C_B(\mathbb{R})$.
Furthermore, we have that $\displaystyle \lim_{n\rightarrow \infty} \mathtt{B}_n^*(f)=f$, pointwise and uniformly for  $f\in C_{UB}(\mathbb{R})$.
\end{theorem}
\begin{proof}
Let $F_1$ and $F_2$ be defined as above. Then we obtain by using Theorem \ref{t1} that
\smallskip
\begin{align*}
\left\vert \mathtt{B}_n^* (f)(x)-f(x)\right\vert&=\left\vert n \int_{-\infty}^{ \infty} \left( \int^{\frac{v+1}{n}}_{\frac{v}{n}} f \left( t\right)dt\right) \Psi(nx-v) dv -\int_{-\infty}^{ \infty} f(x) \Psi(nx-v)dv  \right \vert \\
&=\left\vert n \int_{-\infty}^{ \infty} \left( \int^{\frac{v+1}{n}}_{\frac{v}{n}} f \left( t\right)dt\right) \Psi(nx-v) dv -n \int_{-\infty}^{ \infty} \left(   \int^{\frac{v+1}{n}}_{\frac{v}{n}} f(x)dt \right) \Psi(nx-v)dv \right\vert \\ 
&\leq  n \int_{-\infty}^{ \infty} \left( \int^{\frac{v+1}{n}}_{\frac{v}{n}} \left\vert f(t)-f(x)\right\vert dt
\right) \Psi(nx-v) dv\\
&=  n \int_{-\infty}^{ \infty} \left( \int^{\frac{1}{n}}_{0} \left\vert f\left(t+\frac{v}{n}\right) -f(x)\right\vert dt
\right) \Psi(nx-v) dv \\
&\leq   \int_{F_1} \left( n  \int^{\frac{1}{n}}_{0} \left\vert f\left(t+\frac{v}{n}\right) -f(x)\right\vert dt
\right)  \Psi(nx-v) dv \\
&+\int_{F_2} \left( n  \int^{\frac{1}{n}}_{0} \left\vert f\left(t+\frac{v}{n}\right) -f(x)\right\vert dt
\right) \Psi(nx-v) dv\\
&\leq   \int_{F_1} \left( n  \int^{\frac{1}{n}}_{0}  \omega \left(f,|t|+\frac{1}{n^{\alpha}}\right)dt
\right) \Psi(nx-v) dv +2 \Vert f \Vert_{\infty}\int_{F_2}  \Psi(nx-v)dv\\
&\leq \omega \left(f, \frac{1}{n}+\frac{1}{n^{\alpha}}\right)+ \frac{2\left( q+\frac{1}{q} \right)  \Vert f \Vert_{\infty}}{e^{\beta (n^{1-\alpha}-1)}}
\end{align*}
which gives the desired results. 
\end{proof}
\begin{theorem}
\label{theorem5}
\begin{equation*}
\left\vert \overline{\mathtt{B}_n} (f)(x)-f(x)\right\vert \leq\omega\left(f,\frac{1}{n}+\frac{1}{n^{\alpha}}\right)+ \frac{ 2\left( q+\frac{1}{q} \right) \Vert f \Vert_{\infty}}{e^{ \beta (n^{1-\alpha}-1)}}=\mathtt{E}
\end{equation*} 
and 
\begin{equation*}
\Vert \overline{\mathtt{B}_n}(f)-f\Vert_{\infty} \leq \mathtt{E}
\end{equation*} hold for $f\in C_B(\mathbb{R})$.
Furthermore  we have  $\displaystyle \lim_{n\rightarrow \infty} \overline{\mathtt{B}_n}(f)=f$, pointwise and uniformly for  $f\in C_{UB}(\mathbb{R})$.
\end{theorem}
\begin{proof}
Again, we can write by using  Theorem \ref{t1} that
\begin{align*}
|\overline{\mathtt{B}_n}(f)(x)-f(x)|&=\left\vert \int_{-\infty}^{\infty}  \sum_{s=1}^r w_s f \left( \frac{v}{n}+\frac{s}{nr}\right) \Psi(nx-v) dv-  \int_{-\infty}^{\infty} \left( \sum_{s=1}^r w_s f(x)\right) \Psi(nx-v) dv\right\vert\\
&\leq \int_{-\infty}^{ \infty}  \sum_{s=1}^r w_s\left\vert f \left( \frac{v}{n}+\frac{s}{nr}\right)-f(x)\right\vert  \Psi(nx-v) dv\\
&\leq \int_{F_1}  \sum_{s=1}^r w_s \left\vert f \left( \frac{v}{n}+\frac{s}{nr}\right)-f(x)\right\vert  \Psi(nx-v) dv \\
&+\int_{F_2}  \sum_{s=1}^r w_s  \left\vert f \left( \frac{v}{n}+\frac{s}{nr}\right)-f(x)\right\vert  \Psi(nx-v) dv\\
&\leq \omega \left(f, \frac{1}{n}+\frac{1}{n^{\alpha}}\right)+ \frac{2 \left( q+\frac{1}{q} \right)\Vert f \Vert_{\infty}}{e^{\beta (n^{1-\alpha}-1)}}
\end{align*}
which gives the desired results. 
\end{proof}
\begin{proposition}
\begin{equation*}
\int_{-\infty}^{\infty} |h|^k \Psi(h)dh \leq   \frac{1-e^{-\beta }}{(1+e^{-\beta })}\frac{1}{k+1} +\left( q+\frac{1}{q} \right)\frac{e^{\beta}}{\beta^k} \Gamma(k+1)< \infty 
\end{equation*}
holds for $k \in \mathbb{N}$.
\end{proposition}
\begin{proof} 
We can easily compute that
\begin{align*}
\int_{-\infty}^{\infty} |h|^k \Psi(h)dh&=2\int_{0}^{\infty} h^k  \Psi(h)dh\\
&=2\left(\int_{0}^{1} h^k  \Psi(h)dh +\int_{1}^{\infty} h^k  \Psi(h)dh\right)\\
&\leq 2\left( \frac{1-e^{-\beta }}{2(1+e^{-\beta })}\int_{0}^{1} h^k dh +\int_{1}^{\infty} h^k \frac{1}{2}\left( q+\frac{1}{q} \right) \beta e^{-\beta(h-1)} dh\right)\\
&\leq 2\left( \frac{1-e^{-\beta }}{2(1+e^{-\beta })}\frac{1}{k+1} +\frac{1}{2}\left( q+\frac{1}{q} \right)\beta \int_{1}^{\infty} h^k e^{-\beta(h-1)} dh\right)\\
&\leq 2\left( \frac{1-e^{-\beta }}{2(1+e^{-\beta })}\frac{1}{k+1} +\frac{1}{2} \left( q+\frac{1}{q} \right)\frac{\beta e^{\beta}}{\beta^k} \int_{0}^{\infty} h^k \beta^k e^{-\beta h} \beta dh\right)\\
&\leq  \frac{1-e^{-\beta }}{(1+e^{-\beta })}\frac{1}{k+1} +\left( q+\frac{1}{q} \right)\frac{ e^{\beta}}{\beta^k} \int_{0}^{\infty} x^k e^{-x} dx \\
&\leq  \frac{1-e^{-\beta }}{(1+e^{-\beta })}\frac{1}{k+1} +\left( q+\frac{1}{q} \right)\frac{ e^{\beta}}{\beta^k} \Gamma(k+1)\\
&\leq  \frac{1-e^{-\beta }}{(1+e^{-\beta })}\frac{1}{k+1} +\left( q+\frac{1}{q} \right)\frac{ e^{\beta} k!}{\beta^k} < \infty.
\end{align*}
\end{proof}
\begin{theorem}
\label{theorem6}
\begin{equation}
\label{38}
\omega(\mathtt{B}_n(f),\theta)\leq\omega(f,\theta),~\theta >0
\end{equation}
 holds for $f\in C_B(\mathbb{R})\cup C_U(\mathbb{R})$. 
Furthermore, we have  $\mathtt{B}_n(f)\in C_U (\mathbb{R})$  for  $f\in C_U(\mathbb{R})$ where  $ C_U(\mathbb{R})$ is the set of all uniformly continuous functions on  $\mathbb{R}$.
\end{theorem}
\begin{proof}
Since 
\begin{equation*}
\mathtt{B}_n (f)(x)=\int_{-\infty}^{\infty} f\left( x-\frac{h}{n}\right)\Psi(h)dh,
\end{equation*}
let $x,y \in \mathbb{R}$ and then we can write that 
\begin{equation*}
\mathtt{B}_n (f)(x)-\mathtt{B}_n (f)(y)=\int_{-\infty}^{\infty} \left( f \left( x-\frac{h}{n}\right)-f\left( y-\frac{h}{n} \right)\right)\Psi(h)dh.
\end{equation*}
Hence
\begin{align*}
\left \vert \mathtt{B}_n (f)(x)-\mathtt{B}_n (f)(y)\right\vert &\leq \int_{-\infty}^{\infty} \left \vert f \left( x-\frac{h}{n}\right)-f\left( y-\frac{h}{n} \right)\right \vert \Psi(h)dh\\
& \leq \omega (f,|x-y|)\int_{-\infty}^{\infty}  \Psi(h)dh=\omega(f,|x-y|)
\end{align*} holds and by letting $|x-y|\leq \theta$, $\theta>0 $ then we obtain the desired results. 
\end{proof}
\begin{remark}
It is clear that the equality in  (\ref{38}) holds for $f=identity~map=:id$, and we have that 
\begin{equation*}
|\mathtt{B}_n(id)(x)-\mathtt{B}_n(id)(y)|=|id(x)-id(y)|=|x-y|, 
\end{equation*}
\begin{equation*}
\omega(\mathtt{B}_n(id), \theta)=\omega(id, \theta)=\theta>0. 
\end{equation*}
Also, we have that 
\begin{align*}
\mathtt{B}_n(id)(x)&=\int_{-\infty}^{\infty} \left( x-\frac{h}{n}\right) \Psi(h)dh \\
&=x\int_{-\infty}^{\infty}\Psi(h)dh-\frac{1}{n}\int_{-\infty}^{\infty} h \Psi(h)dh=x-\frac{1}{n}\int_{-\infty}^{\infty} h \Psi(h)dh
\end{align*}
and for fixed $x \in \mathbb{R}$,
\begin{align*}
|\mathtt{B}_n(id)(x)|&\leq |x|+\frac{1}{n} \int_{-\infty}^{\infty} |h| \Psi(h)dh \\
&\leq|x|+\frac{1}{n} \left[\frac{1-e^{-\beta }}{2(1+e^{-\beta })} +\frac{ \left( q +\frac{1}{q}\right)e^{\beta }}{\beta} \right]< \infty.
\end{align*}
It is obvious that $id \in C_U (\mathbb{R}).$
\end{remark}
\begin{theorem}
\label{theorem7}
\begin{equation}
\label{42}
\omega (\mathtt{B}^*_n(f),\theta)\leq\omega(f,\theta),~\theta>0
\end{equation} holds for $f\in C_B(\mathbb{R})\cup C_U(\mathbb{R})$.
Furthermore, we have $\mathtt{B}^*_n(f)\in C_U(\mathbb{R})$ for $f\in C_U(\mathbb{R})$.
\end{theorem}
\begin{proof}
Notice that
\begin{equation*}
\mathtt{B}^*_n(f)(x)=n \int_{-\infty}^{ \infty} \left( \int^{\frac{1}{n}}_0 f \left( t+\left(x-\frac{h}{n}\right)\right)dt\right) \Psi(h)dh
\end{equation*}
 and
\begin{equation*}
\mathtt{B}^*_n(f)(y)=n \int_{-\infty}^{ \infty} \left( \int^{\frac{1}{n}}_0 f \left( t+\left(y-\frac{h}{n}\right)\right)dt\right)\Psi(h)dh
\end{equation*}
for $ x,~y \in \mathbb{R} $.\\
Hence $\left\vert \mathtt{B}^*_n (f)(x)-\mathtt{B}^*_n (f)(y)\right\vert$
\begin{align*}
\leq& n \int_{-\infty}^{\infty} \left( \int^{\frac{1}{n}}_0 \left \vert f \left( t+\left(x-\frac{h}{n}\right)\right)- f \left( t+\left(y-\frac{h}{n}\right)\right) \right \vert dt\right) \Psi(h)dh \\
\leq &\omega (f,|x-y|)\int_{-\infty}^{\infty} \Psi(h)dh=\omega(f,|x-y|)
\end{align*}
 and again implies  the desired results. 
Furthermore, 
\begin{equation*}
|\mathtt{B}^*_n(id)(x)-\mathtt{B}^*_n(id)(y)|=|x-y|
\end{equation*}
and 
\begin{align*}
|\mathtt{B}^*_n(id)(x)|&\leq n \int_{-\infty}^{ \infty} \left( \int^{\frac{1}{n}}_0 \left(|t|+|x|+\frac{|h|}{n}\right)dt\right) \Psi(h)dh \\& \leq \int_{-\infty}^{\infty} \left(\frac{1}{n}+|x|+\frac{|h|}{n}\right) \Psi(h)dh\\
&= \frac{1}{n}+|x|+\frac{1}{n} \int_{-\infty}^{\infty} |h|\Psi(h)dh < \infty.
\end{align*}
This proves the attainability of (\ref{42}).
\end{proof}
\begin{theorem}
\label{theorem8}
\begin{equation}
\label{46}
\omega( \overline{\mathtt{B}_n}(f),\theta)\leq\omega(f,\theta),~\theta>0
\end{equation}
holds for $f\in C_B(\mathbb{R}) \cup C_U(\mathbb{R}) $.
Furthermore, we have  $ \overline{B_n}(f)\in C_U(\mathbb{R})$ for $f\in C_U(\mathbb{R})$.
\end{theorem} 
\begin{proof}
Notice that
\begin{equation*}
\overline{\mathtt{B}_n}(f)(x)=\int_{-\infty}^{\infty} \left( \sum_{s=1}^r w_s f \left(\left(x- \frac{h}{n} \right)+\frac{s}{nr}\right)\right) \Psi(h)dh
\end{equation*}
 and
\begin{equation*}
\overline{\mathtt{B}_n}(f)(y)=\int_{-\infty}^{\infty} \left( \sum_{s=1}^r w_s f \left(\left(y- \frac{h}{n} \right)+\frac{s}{nr}\right)\right)\Psi(h)dh
\end{equation*}
for $x,~y \in \mathbb{R}$. \\
Hence  $\left\vert \overline{\mathtt{B}_n} (f)(x)-\overline{\mathtt{B}_n} (f)(y)\right\vert$
\begin{align*}
=&\left \vert \int_{-\infty}^{ \infty}  \sum_{s=1}^r w_s \left[ f \left(\left(x- \frac{h}{n} \right)+\frac{s}{nr}\right)-f \left(\left(y- \frac{h}{n} \right)+\frac{s}{nr}\right)\right] \Psi(h)dh \right \vert\\
\leq &  \int_{-\infty}^{ \infty}  \sum_{s=1}^r w_s\left \vert f \left(\left(x- \frac{h}{n} \right)+\frac{s}{nr}\right)-f \left(\left(y- \frac{h}{n} \right)+\frac{s}{nr}\right)\right \vert \Psi(h)dh \\
\leq &\omega (f,|x-y|)\int_{-\infty}^{\infty}\Psi(h)dh=\omega(f,|x-y|)
\end{align*}
 and this implies the desired results.
Furthermore,
\begin{equation*}
|\overline{\mathtt{B}_n}(id)(x)-\overline{\mathtt{B}_n}(id)(y)|=|x-y|
\end{equation*}
and 
\begin{align*}
|\overline{\mathtt{B}_n}(id)(x)|&\leq \int_{-\infty}^{ \infty} \left( \sum_{s=1}^r w_s  \left(|x|+ \frac{|h|}{n}+ \frac{1}{n} \right)\right)\Psi(h)dh\\
&\leq \int_{-\infty}^{ \infty} \left( |x|+ \frac{|h|}{n}+ \frac{1}{n}\right)\Psi(h)dh= \frac{1}{n}+|x|+\frac{1}{n}\int_{-\infty}^{\infty} |h|dh< \infty
\end{align*}
proving the attainability of (\ref{46}).
\end{proof}
\begin{remark}
Fix $s \in \mathbb{N}$ and let  $f \in C^{(s)}(\mathbb{R})$ such that $f^{(k)} \in C_B (\mathbb{R})$ for $k=1,2,\ldots, s$. Then we can write that 
\begin{equation*}
\mathtt{B}_n(f)(x)=\int_{-\infty}^{ \infty} f\left(x-\frac{h}{n}\right) \Psi(h)dh
\end{equation*}
and with the use of  Leibniz's rule 
\begin{align*}
\frac{\partial^s \mathtt{B}_n(x)}{\partial x^s} & =\int_{-\infty}^{\infty} f^{(s)}\left(x-\frac{h}{n}\right) \Psi(h)dh\\
&=\int_{-\infty}^{ \infty} f^{(s)}\left(\frac{v}{n}\right) \Psi(nx-v)dv=\mathtt{B}_n( f^{(s)})(x),~ for~all~x \in \mathbb{R}.
\end{align*}
\end{remark}
Clearly the followings also hold:
\begin{equation*}
\frac{\partial^s \mathtt{B}^*_n(f)(x)}{\partial x^s}=\mathtt{B}^*_n( f^{(s)})(x);~\frac{\partial^s \overline{ \mathtt{B}_n} (f)(x)}{\partial x^s}= \overline{ \mathtt{B}_n}( f^{(s)})(x).
\end{equation*}
\begin{remark}
Fix $s\in \mathbb{N}$  and let  $f \in C^{(s)} (\mathbb{R})$ such that  $f^{(k)} \in C_B (\mathbb {R})$ for $k=1,2,\ldots, s $. We  derive that  
\begin{align*}
\left( \mathtt{B}_n(f)\right)^{(k)}(x)&= \mathtt{B}_n(f^{(k)})(x),\\
\left(\mathtt{B}^*_n(f)\right)^{(k)}(x)&= \mathtt{B}^*_n(f^{(k)})(x),\\
\left(\overline{ \mathtt{B}_n}(f)\right)^{(k)}(x)&=\overline{ \mathtt{B}_n}(f^{(k)})(x),~ for~all~x\in \mathbb{R} ~and~ k=1,2,\ldots, s.
\end{align*}
\end{remark}
\noindent Then we have the followings under the assumptions of  $0<\alpha<1$ and $n \in \mathbb{N}$ such that $n^{1-\alpha}>2$.
\begin{theorem}
Let  $f^{(k)} \in C_B(\mathbb{R})$ for $k=1, 2, \cdots, s \in \mathbb{N}$. Then we have the followings:
\begin{itemize}
\item[(i)]
\begin{equation*}
\left| ( \mathtt{B}_n (f))^{(k)}(x)-f^{(k)}(x)\right|\leq\omega\left(f^{(k)},\frac{1}{n^{\alpha}}\right)+ \frac{2\left( q+\frac{1}{q} \right)  \Vert f^{(k)} \Vert_{\infty}}{e^{ \beta (n^{1-\alpha}-1)}}=:\mathtt{T}_k
\end{equation*}
and 
\begin{equation*}
\Vert (\mathtt{B}_n(f))^{(k)}-f^{(k)}\Vert_{\infty} \leq  \mathtt{T}_k,
\end{equation*}
\item[(ii)]
\begin{equation*}
\left\vert (\mathtt{B}_n^* (f))^{(k)}(x)-f^{(k)}(x)\right\vert \leq\omega\left(f^{(k)},\frac{1}{n}+\frac{1}{n^{\alpha}}\right)+ \frac{2\left( q+\frac{1}{q} \right)\Vert f^{(k)} \Vert_{\infty}}{e^{\beta (n^{1-\alpha}-1)}}=: \mathtt{E}_k
\end{equation*} 
and 
\begin{equation*}
\Vert (\mathtt{B}_n^*(f))^{(k)}-f^{(k)} \Vert_{\infty} \leq \mathtt{E}_k,
\end{equation*}
\begin{equation*}
\end{equation*}
\item[(iii)]
\begin{equation*}
\left\vert (\overline{\mathtt{B}_n} (f))^{(k)} (x)-f^{(k)}(x)\right\vert \leq\omega\left(f^{(k)},\frac{1}{n}+\frac{1}{n^{\alpha}}\right)+ \frac{ 2\left( q+\frac{1}{q} \right) \Vert f^{(k)} \Vert_{\infty}}{e^{\beta (n^{1-\alpha}-1)}}=\mathtt{E}_k
\end{equation*} 
and 
\begin{equation*}
\Vert (\overline{\mathtt{B}_n}(f))^{(k)}-f^{(k)}\Vert_{\infty} \leq \mathtt{E}_k.
\end{equation*}
\end{itemize}
\end{theorem}
\begin{proof}
Taking into consideration of Theorems \ref{theorem3}, \ref{theorem4}, \ref{theorem5}, we immediately obtain the proof.
\end{proof}
Now, we present a similar result on the preservation of global smoothness.
\begin{theorem}
Let $f^{(k)} \in C_B(\mathbb{R}) \cup C_U (\mathbb{R})$, for $k=0, 1, \cdots, s \in \mathbb{N}$. Then we have the followings:
\begin{itemize}
\item[(i)]
\begin{equation*}
\omega \left( (\mathtt{B}_n(f))^{(k)}, \theta \right)\leq \omega \left( f^{(k)}, \theta \right), ~\theta>0,
\end{equation*}
also $\left( \mathtt{B}_n(f) \right)^{(k)} \in C_U(\mathbb{R})$ when   $f^{(k)}\in C_U(\mathbb{R})$,
\item[(ii)]
\begin{equation*}
\omega \left( (\mathtt{B}^*_n(f))^{(k)}, \theta \right)\leq \omega \left( f^{(k)}, \theta \right), ~\theta>0,
\end{equation*}
also $\left( \mathtt{B}^*_n(f) \right)^{(k)} \in C_U(\mathbb{R})$ when   $f^{(k)}\in C_U(\mathbb{R})$,
\item[(iii)]
\begin{equation*}
\omega \left( (\overline{\mathtt{B}_n} (f))^{(k)}, \theta \right)\leq \omega \left( f^{(k)}, \theta \right), ~\theta>0,
\end{equation*}
also   $\left( \overline{\mathtt{B}_n}(f) \right)^{(k)} \in C_U(\mathbb{R})$ when   $f^{(k)}\in C_U(\mathbb{R})$.
\end{itemize}
\end{theorem}
\begin{proof}
Taking into consideration of Theorems \ref{theorem6}, \ref{theorem7}, \ref{theorem8}, we immediately obtain the proof.
\end{proof}
Now, we  present results dealing with the improvement of the rate of convergence of our operators by assuming the differentiability of functions.
\begin{theorem}
\label{10}
If  $0<\alpha<1$, $n \in \mathbb{N}:n^{1-\alpha}>2$, $x \in \mathbb{R}$, $f \in C^N (\mathbb{R})$, $N \in \mathbb{N}$, such that $f^{(N)} \in C_B(\mathbb{R})$, then the followings hold:
\begin{itemize}
\item[(i)] \begin{align*}
\left| ( \mathtt{B}_n (f))(x)-f(x) - \sum_{k=1}^N \frac{f^{(k)}(x)}{k!} \left(\mathtt{A}_n\left( (\cdot -x)^k \right)\right)(x)\right|
\end{align*}
\begin{equation*}
\leq \frac{\omega \left( f^{(N)}, \frac{1}{n^{\alpha}} \right)}{n^{\alpha N}N!}+\frac{2^{N+2}  \Vert f^{(N)} \Vert_{\infty}e^{\beta } \left( q+\frac{1}{q} \right)}{n^N \beta ^N } e^{ \frac{-\beta  n^{1-\alpha}}{2}}\rightarrow 0, ~as~ n\rightarrow  \infty,
\end{equation*}
\item[(ii)]
if $ f^{(k)}(x)=0$, $k=1, 2, \cdots, N$  then
\begin{equation*}
\vert \mathtt{B}_n(f)(x)-f(x)\vert \leq \frac{\omega \left( f^{(N)}, \frac{1}{n^{\alpha}} \right)}{n^{\alpha N}N!}+ \frac{2^{N+2}  \Vert f^{(N)} \Vert_{\infty}e^{\beta } \left( q+\frac{1}{q} \right)}{n^N \beta ^N } e^{ \frac{-\beta  n^{1-\alpha}}{2}},
\end{equation*}
\item[(iii)]
\begin{align*}
\vert \mathtt{B}_n(f)(x)-f(x)\vert& \leq \sum_{k=1}^N \frac{\left \vert f^{(k)}(x) \right \vert}{k!} \frac{1}{n^k} \left[  \frac{1-e^{-\beta }}{(1+e^{-\beta })}\frac{1}{k+1} +\left( q+\frac{1}{q} \right)\frac{ e^{\beta} k!}{\beta^k} \right]\\
&+  \frac{\omega \left( f^{(N)}, \frac{1}{n^{\alpha}} \right)}{n^{\alpha N}N!}+ \frac{2^{N+2}  \Vert f^{(N)} \Vert_{\infty}e^{\beta } \left( q+\frac{1}{q} \right)}{n^N \beta ^N } e^{ \frac{-\beta  n^{1-\alpha}}{2}},
\end{align*}
\item[(iv)]
\begin{align*}
\Vert \mathtt{B}_n(f)-f\Vert_{\infty}& \leq \sum_{k=1}^N \frac{\left \Vert f^{(k)} \right \Vert_{\infty}}{k!} \frac{1}{n^k} \left[  \frac{1-e^{-\beta }}{(1+e^{-\beta })}\frac{1}{k+1} +\left( q+\frac{1}{q} \right)\frac{ e^{\beta} k!}{\beta^k} \right]\\
&+  \frac{\omega \left( f^{(N)}, \frac{1}{n^{\alpha}} \right)}{n^{\alpha N}N!}+ \frac{2^{N+2}  \Vert f^{(N)} \Vert_{\infty}B^{ \beta } \left( q+\frac{1}{q} \right)}{n^N \beta ^N (\ln B)^N} B^{ \frac{-\beta  n^{1-\alpha}}{2}}.
\end{align*}
\end{itemize}
\end{theorem}
\begin{proof}
Since it is already known that 
\begin{equation*}
f \left(  \frac{v}{n}\right)=\sum_{k=0}^{N} \frac{f^{(k)} (x)}{k!} \left( \frac{v}{n}-x \right)^k+\int_x^{\frac{v}{n}} \left( f^{(N)}(t)-f^{(N)}(x)\right)\frac{ \left(\frac{v}{n}-t \right)^{N-1}}{(N-1)!}dt
\end{equation*}
and 
\begin{align*}
f \left(  \frac{v}{n}\right) \Psi(nx-v)& =\sum_{k=0}^{N} \frac{f^{(k)} (x)}{k!}  \Psi(nx-v) \left( \frac{v}{n}-x \right)^k\\
&+ \Psi(nx-v)\int_x^{\frac{v}{n}} \left( f^{(N)}(t)-f^{(N)}(x)\right)\frac{ \left(\frac{v}{n}-t \right)^{N-1}}{(N-1)!}dt,
\end{align*}
 we can write that
\begin{align*}
\mathtt{B}_n(f)(x)&=\int_{-\infty}^{\infty} f \left(  \frac{v}{n}\right) \Psi(nx-v)dv \\
&=\sum_{k=0}^{N} \frac{f^{(k)} (x)}{k!} \int_{-\infty}^{\infty} \Psi(nx-v) \left( \frac{v}{n}-x \right)^k dv\\
&+\int_{-\infty}^{\infty} \Psi(nx-v) \left( \int_x^{\frac{v}{n}} \left( f^{(N)}(t)-f^{(N)}(x)\right)\frac{ \left(\frac{v}{n}-t \right)^{N-1}}{(N-1)!}dt\right) dv
\end{align*}
and
\begin{align*}
 \mathtt{B}_n(f)(x)-f(x)&=\sum_{k=1}^{N} \frac{f^{(k)} (x)}{k!} \int_{-\infty}^{\infty} \Psi(nx-v) \left( \frac{v}{n}-x \right)^k dv\\
&+\int_{-\infty}^{\infty} \Psi(nx-v) \left( \int_x^{\frac{v}{n}} \left( f^{(N)}(t)-f^{(N)}(x)\right)\frac{ \left(\frac{v}{n}-t \right)^{N-1}}{(N-1)!}dt\right) dv.
\end{align*}
Let 
\begin{equation*}
R_n(x):=\int_{-\infty}^{\infty} \Psi(nx-v) \left( \int_x^{\frac{v}{n}} \left( f^{(N)}(t)-f^{(N)}(x)\right)\frac{ \left(\frac{v}{n}-t \right)^{N-1}}{(N-1)!}dt\right) dv
\end{equation*}
and 
\begin{equation*}
\gamma(v):= \int_x^{\frac{v}{n}} \left( f^{(N)}(t)-f^{(N)}(x)\right)\frac{ \left(\frac{v}{n}-t \right)^{N-1}}{(N-1)!}dt.
\end{equation*}
Now, we have two cases: \begin{itemize} \item[1.] 
$  \left \vert \frac{v}{n}-x  \right \vert< \frac{1}{n^\alpha},$\\
\item[2.] $  \left \vert \frac{v}{n}-x  \right \vert \geq \frac{1}{n^\alpha}.$
\end{itemize} 
Let us start with Case 1.\\
Case 1 (i): When $ \frac{v}{n}\geq x$ we have 
\begin{align*}
|\gamma(v)|&
 \leq \omega \left( f^{(N)}, \frac{1}{n^{\alpha}} \right) \left(\int_x^{\frac{v}{n}} \frac{ \left(\frac{v}{n}-t \right)^{N-1}}{(N-1)!}dt\right) \\
&=\omega \left( f^{(N)}, \frac{1}{n^{\alpha}} \right) \frac{ \left(\frac{v}{n}-x\right)^{N}}{N!}\leq \omega \left( f^{(N)}, \frac{1}{n^{\alpha}} \right) \frac{1}{n^{\alpha N}N!}.
\end{align*}
Case 1 (ii): When $ \frac{v}{n} < x$, we have  
\smallskip
\begin{align*}
|\gamma(v)|& \leq
\omega \left( f^{(N)}, \frac{1}{n^{\alpha}} \right) \frac{ \left(x-\frac{v}{n}\right)^{N}}{N!}\leq \omega \left( f^{(N)}, \frac{1}{n^{\alpha}} \right) \frac{ 1}{n^{\alpha N}N!}.
\end{align*}
Therefore $ |\gamma(v)| \leq \omega \left( f^{(N)}, \frac{1}{n^{\alpha}} \right) \frac{ 1}{n^{\alpha N}N!}$ holds whenever  $  \left \vert \frac{v}{n}-x  \right \vert< \frac{1}{n^\alpha} $.
Then we have that
\begin{equation*}
\left \vert \int_{ \left \vert \frac{v}{n}-x  \right \vert< \frac{1}{n^\alpha} } \Psi(nx-v) \left( \int_x^{\frac{v}{n}} \left( f^{(N)}(t)-f^{(N)}(x)\right)\frac{ \left(\frac{v}{n}-t \right)^{N-1}}{(N-1)!}dt\right) dv \right \vert 
\end{equation*}
\begin{equation*}
\leq \int_{ \left \vert \frac{v}{n}-x  \right \vert< \frac{1}{n^\alpha} } \Psi(nx-v)  |\gamma(v)| dv \leq  \omega \left( f^{(N)}, \frac{1}{n^{\alpha}} \right) \frac{ 1}{n^{\alpha N}N!}.
\end{equation*}
Case 2: In this case, we can write that
\begin{align*}
&\left \vert \int_{ \left \vert \frac{v}{n}-x  \right \vert\geq \frac{1}{n^\alpha} } \Psi(nx-v) \left( \int_x^{\frac{v}{n}} \left( f^{(N)}(t)-f^{(N)}(x)\right)\frac{ \left(\frac{v}{n}-t \right)^{N-1}}{(N-1)!}dt\right) dv \right \vert \\
&\leq  \int_{ \left \vert \frac{v}{n}-x  \right \vert\geq \frac{1}{n^\alpha} } \Psi(nx-v)\left \vert \int_x^{\frac{v}{n}} \left( f^{(N)}(t)-f^{(N)}(x)\right)\frac{ \left(\frac{v}{n}-t \right)^{N-1}}{(N-1)!}dt\right \vert dv \\
& =\int_{ \left \vert \frac{v}{n}-x  \right \vert\geq \frac{1}{n^\alpha} } \Psi(nx-v) |\gamma(v)|:= \mathtt{P}.
\end{align*}
If $\frac{v}{n} \geq x$ then $ |\gamma(v)|\leq 2\left \Vert f^{(N)} \right \Vert_{\infty} \frac{ \left(\frac{v}{n}-x \right)^{N}}{N!}  $ and 
if  $\frac{v}{n} < x$ then $ |\gamma(v)|\leq 2\left \Vert f^{(N)} \right \Vert_{\infty} \frac{ \left(x-\frac{v}{n} \right)^{N}}{N!}  $. So by combining them, we can also write  $ |\gamma(v)|\leq 2\left \Vert f^{(N)} \right \Vert_{\infty} \frac{ \left \vert \frac{v}{n}-x \right \vert^{N}}{N!} $
and by using Theorem \ref{t1} that
\begin{align*}
\mathtt{P}&\leq   \frac{2\left \Vert f^{(N)} \right \Vert_{\infty}}{N!}  \int_{ \left \vert \frac{v}{n}-x  \right \vert\geq \frac{1}{n^\alpha} } \Psi(nx-v) \left \vert x-\frac{v}{n} \right \vert^N dv\\
&= \frac{2\left \Vert f^{(N)} \right \Vert_{\infty}}{n^N N!}  \int_{ \left \vert nx-v \right \vert \geq n^{1-\alpha }} \Psi(|nx-v|) \left \vert nx-v \right \vert^N dv\\
&\leq \frac{2\left \Vert f^{(N)} \right \Vert_{\infty}}{n^N N!} \left(q +\frac{1}{q} \right)\beta   \int_{n^{1-\alpha}}^{\infty} e^{-\beta (x-1)} x^N dx\\
&= \frac{2\left \Vert f^{(N)} \right \Vert_{\infty}}{n^N N!} \left(q +\frac{1}{q} \right) \frac{ e^\beta }{\beta ^N} \int_{n^{1-\alpha}}^{\infty} e^{-\beta  x} (\beta  x)^N (\beta  dx)\\
& =\frac{2\left \Vert f^{(N)} \right \Vert_{\infty}}{n^N N!} \left(q +\frac{1}{q} \right) \frac{e^\beta }{\beta^N}  \int_{\beta  n^{1-\alpha}}^{\infty} e^{-t} t^N dt\\
&\leq \frac{2\left \Vert f^{(N)} \right \Vert_{\infty}}{n^N N!} \left(q +\frac{1}{q} \right) \frac{ e^\beta 2^N N!}{\beta ^N }  \int_{\beta  n^{1-\alpha}}^{\infty} e^{-t} e^{\frac{t}{2}}dt\\
&=\frac{2^{N+1}\left \Vert f^{(N)} \right \Vert_{\infty} \left(q +\frac{1}{q} \right)}{n^N } \frac{ e^\beta }{\beta ^N } \int_{\beta  n^{1-\alpha}}^{\infty}  e^{-\frac{t}{2}} dt\\
&=\frac{2^{N+1}\left \Vert f^{(N)} \right \Vert_{\infty} \left(q +\frac{1}{q} \right)}{n^N}  \frac{ e^\beta }{\beta ^N } (-2)e^{-\frac{t}{2}} \bigg\vert^{\infty}_{\beta  n^{1-\alpha}}\\
&=\frac{2^{N+2}  \Vert f^{(N)} \Vert_{\infty}e^{\beta } \left( q+\frac{1}{q} \right)}{n^N \beta ^N } e^{ \frac{-\beta  n^{1-\alpha}}{2}}.
\end{align*}
Then we get that  
\begin{equation*}
\left \vert \int_{ \left \vert \frac{v}{n}-x  \right \vert\geq \frac{1}{n^\alpha} } \Psi(nx-v) \left( \int_x^{\frac{v}{n}} \left( f^{(N)}(t)-f^{(N)}(x)\right)\frac{ \left(\frac{v}{n}-t \right)^{N-1}}{(N-1)!}dt\right) dv \right \vert
\end{equation*}
\begin{equation*}
\leq \frac{2^{N+2}  \Vert f^{(N)} \Vert_{\infty}e^{\beta } \left( q+\frac{1}{q} \right)}{n^N \beta ^N } e^{ \frac{-\beta  n^{1-\alpha}}{2}}\rightarrow0,~n\rightarrow \infty.
\end{equation*}
Finally we conclude that, 
\begin{align*}
\left \vert R_n(x) \right \vert &\leq \left \vert \int_{ \left \vert \frac{v}{n}-x  \right \vert< \frac{1}{n^\alpha} } \Psi(nx-v) \left( \int_x^{\frac{v}{n}} \left( f^{(N)}(t)-f^{(N)}(x)\right)\frac{ \left(\frac{v}{n}-t \right)^{N-1}}{(N-1)!}dt\right) dv \right \vert \\
&+\left \vert \int_{ \left \vert \frac{v}{n}-x  \right \vert\geq \frac{1}{n^\alpha} } \Psi(nx-v) \left( \int_x^{\frac{v}{n}} \left( f^{(N)}(t)-f^{(N)}(x)\right)\frac{ \left(\frac{v}{n}-t \right)^{N-1}}{(N-1)!}dt\right) dv \right \vert \\
&\leq \omega \left( f^{(N)}, \frac{1}{n^{\alpha}} \right) \frac{ 1}{n^{\alpha N}N!}+ \frac{2^{N+2}  \Vert f^{(N)} \Vert_{\infty}e^{\beta } \left( q+\frac{1}{q} \right)}{n^N \beta ^N } e^{ \frac{-\beta  n^{1-\alpha}}{2}}\rightarrow0,~n\rightarrow \infty.
\end{align*}
Now letting $k=1, 2, \cdots, N$ and $h=nx-v$ then we obtain that
\begin{align*}
\left \vert \mathtt{B}_n \left( ( \cdot-x)^k\right)(x) \right \vert& =\left \vert  \int_{-\infty}^{\infty} \Psi(nx-v) \left( \frac{v}{n}-x \right)^k dv \right \vert\\
& \leq \int_{-\infty}^{\infty}  \left \vert \frac{v}{n}-x  \right \vert^k \Psi(nx-v)dv\\
& =\frac{1}{n^k} \int_{-\infty}^{\infty}  \left \vert nx-v \right \vert^k \Psi(nx-v)dv=\frac{1}{n^k} \int_{-\infty}^{\infty}  \left \vert h \right \vert^k \Psi(h)dh\\
&\leq \frac{1}{n^k} \left[  \frac{1-e^{-\beta }}{(1+e^{-\beta })}\frac{1}{k+1} +\left( q+\frac{1}{q} \right)\frac{ e^{\beta} k!}{\beta^k} \right]\rightarrow 0, ~as ~n\rightarrow \infty.
\end{align*}
Hence we complete the proof. 
\end{proof}
\begin{theorem}
\label{11}
If $0<\alpha<1$, $n \in \mathbb{N}:n^{1-\alpha}>2$, $x \in \mathbb{R}$, $f \in C^{N}(\mathbb{R})$, $N \in \mathbb{N}$ such that $f^{(N)} \in C_B(\mathbb{R})$, then the  followings hold:
\begin{itemize}
\item[(i)]
\begin{align*}
\left| \mathtt{B}^*_n (f)(x)-f(x) - \sum_{k=1}^N \frac{f^{(k)}(x)}{k!} \left(\mathtt{A}^*_n\left( (\cdot -x)^k \right) \right)(x)\right|
\end{align*}
\begin{align*}
&\leq  \omega \left(f^{(N)}, \frac{1}{n}+\frac{1}{n^{\alpha}}\right) \frac{\left(  \frac{1}{n}+\frac{1}{n^{\alpha}}  \right)^N}{N!}\\
&+ \frac{ 2^N\left \Vert f^{(N)} \right \Vert_{\infty}}{n^N N!} \left( q+\frac{1}{q} \right)e^{\beta } e^{-\frac{\beta  n^{1-\alpha}}{2}} \left[1 +  \frac{ 2^{N+1} N!}{\beta ^{N}}  \right]\rightarrow 0,~n\rightarrow \infty,
\end{align*}
\item[(ii)]
if $ f^{(k)}(x)=0$, $k=1, 2, \cdots, N$ then 
\begin{align*}
\vert \mathtt{B}^*_n(f)(x)-f(x)\vert &\leq  \omega \left(f^{(N)}, \frac{1}{n}+\frac{1}{n^{\alpha}}\right) \frac{\left(   \frac{1}{n}+\frac{1}{n^{\alpha}} \right)^N}{ N!}\\
&+ \frac{ 2^N\left \Vert f^{(N)} \right \Vert_{\infty}}{n^N N!} \left( q+\frac{1}{q} \right)e^{\beta } e^{-\frac{\beta  n^{1-\alpha}}{2}} \left[1 +  \frac{ 2^{N+1} N!}{\beta ^{N}}  \right],
\end{align*}
\item[(iii)]
\begin{align*}
\vert \mathtt{B}^*_n(f)(x)-f(x)\vert& \leq \sum_{k=1}^N \frac{\left \vert f^{(k)}(x)\right \vert}{k!} \frac{2^{k-1}}{n^k} \left[1+   \frac{1-e^{-\beta }}{(1+e^{-\beta })}\frac{1}{k+1} +\left( q+\frac{1}{q} \right)\frac{ e^{\beta} k!}{\beta^k} \right] \\
&+  \omega \left(f^{(N)}, \frac{1}{n}+\frac{1}{n^{\alpha}}\right) \frac{\left(   \frac{1}{n}+\frac{1}{n^{\alpha}} \right)^N}{N!}\\
&+\frac{ 2^N\left \Vert f^{(N)} \right \Vert_{\infty}}{n^N N!} \left( q+\frac{1}{q} \right)e^{\beta } e^{-\frac{\beta  n^{1-\alpha}}{2}} \left[1 +  \frac{ 2^{N+1} N!}{\beta ^{N}}  \right],
\end{align*}
\item[(iv)]
\begin{align*}
\Vert \mathtt{B}^*_n(f)-f\Vert_{\infty}& \leq \sum_{k=1}^N \frac{\left\Vert f^{(k)}\right \Vert_{\infty}}{k!} \frac{2^{k-1}}{n^k} \left[1+  \frac{1-e^{-\beta }}{(1+e^{-\beta })}\frac{1}{k+1} +\left( q+\frac{1}{q} \right)\frac{ e^{\beta} k!}{\beta^k} \right] \\
&+ \omega \left(f^{(N)}, \frac{1}{n}+\frac{1}{n^{\alpha}}\right) \frac{\left(   \frac{1}{n}+\frac{1}{n^{\alpha}} \right)^N}{N!}\\
&+\frac{ 2^N\left \Vert f^{(N)} \right \Vert_{\infty}}{n^N N!} \left( q+\frac{1}{q} \right)e^{\beta } e^{-\frac{\beta  n^{1-\alpha}}{2}} \left[1 +  \frac{ 2^{N+1} N!}{\beta ^{N}}  \right].
\end{align*}
\end{itemize}
\end{theorem}
\begin{proof}
Since it is already known that
\begin{equation*}
f \left(  t+\frac{v}{n}\right)=\sum_{k=0}^{N} \frac{f^{(k)} (x)}{k!} \left( t+\frac{v}{n}-x \right)^k
+\int_x^{t+\frac{v}{n}} \left( f^{(N)}(s)-f^{(N)}(x)\right)\frac{ \left(t+\frac{v}{n}-s \right)^{N-1}}{(N-1)!}ds
\end{equation*}
and
\begin{align*}
\int_{0}^{\frac{1}{n}}f \left(  t+\frac{v}{n}\right)dt& =\sum_{k=0}^{N} \frac{f^{(k)} (x)}{k!} \int_{0}^{\frac{1}{n}} \left( t+\frac{v}{n}-x \right)^kdt+\int_{0}^{\frac{1}{n}}\left(\int_x^{t+\frac{v}{n}} \left( f^{(N)}(s)-f^{(N)}(x)\right)\frac{ \left(t+\frac{v}{n}-s \right)^{N-1}}{(N-1)!}ds\right)dt,
\end{align*}
we can write that 
\begin{align*}
&n \int_{-\infty}^{\infty} \left( \int_{0}^{\frac{1}{n}}f \left( t+ \frac{v}{n}\right)dt \right) \Psi(nx-v)dv\\
&=\sum_{k=0}^{N} \frac{f^{(k)} (x)}{k!} n \int_{-\infty}^{\infty} \left( \int_{0}^{\frac{1}{n}}f \left( t+ \frac{v}{n}-x \right)^k dt \right) \Psi(nx-v)dv\\
&+n \int_{-\infty}^{\infty}  \left( \int_{0}^{\frac{1}{n}} \left(  \int_x^{t+\frac{v}{n}} \left( f^{(N)}(s)-f^{(N)}(x)\right)\frac{ \left(t+\frac{v}{n}-s \right)^{N-1}}{(N-1)!}ds\right)dt \right)\Psi(nx-v)dv
\end{align*}
and 
\begin{align*}
 \mathtt{B}^*_n(f)(x)-f(x)&=\sum_{k=1}^{N} \frac{f^{(k)} (x)}{k!} \mathtt{A}^*_n \left( (\cdot -x)^k \right)(x)+R_n(x).
\end{align*}
Here, let
\begin{equation*}
R_n(x):= n \int_{-\infty}^{\infty}  \left( \int_{0}^{\frac{1}{n}} \left(  \int_x^{t+\frac{v}{n}} \left( f^{(N)}(s)-f^{(N)}(x)\right)\frac{ \left(t+\frac{v}{n}-s \right)^{N-1}}{(N-1)!}ds\right)dt\right) \Psi(nx-v)dv
\end{equation*}
and 
\begin{equation*}
\gamma(v):= n\int_{0}^{\frac{1}{n}} \left(  \int_x^{t+\frac{v}{n}} \left( f^{(N)}(s)-f^{(N)}(x)\right)\frac{ \left(t+\frac{v}{n}-s \right)^{N-1}}{(N-1)!}ds\right)dt.
\end{equation*}
Now we have two cases : \begin{itemize}
\item[1.] $ \left \vert \frac{v}{n}-x  \right \vert < \frac{1}{n^\alpha}, $
\item[2.] $ \left \vert \frac{v}{n}-x  \right \vert \geq \frac{1}{n^\alpha}. $
\end{itemize} 
Let us start with Case 1. 
\begin{itemize}
\item[Case 1 (i):]
 When $ t+\frac{v}{n} \geq x $, we have
\begin{align*}
|\gamma(v)| &\leq n \int_{0}^{\frac{1}{n}} \left(  \int_x^{t+\frac{v}{n}} \left\vert f^{(N)}(s)-f^{(N)}(x)\right \vert \frac{ \left( t+\frac{v}{n}-s \right)^{N-1}}{(N-1)!}ds \right)dt \\
& \leq n \int_{0}^{ \frac{1}{n}} \omega \left( f^{(N)}, |t|+\left \vert \frac{v}{n}-x \right \vert \right) \left(  \int_x^{t+\frac{v}{n}}  \frac{ \left(t+\frac{v}{n}-s \right)^{N-1}}{(N-1)!}ds \right)dt\\
&\leq \omega \left( f^{(N)}, \frac{1}{n}+\frac{1}{n^{\alpha}} \right) n  \int_{0}^{\frac{1}{n}}  \frac{\left( |t|+\left \vert \frac{v}{n}-x \right \vert \right)^N }{N!} dt \\
&\leq \frac{\omega \left(f^{(N)}, \frac{1}{n}+\frac{1}{n^{\alpha}}\right)}{N!} \left( \frac{1}{n}+\frac{1}{n^{\alpha}}\right)^N.
\end{align*}
\item[Case 1 (ii):]
When $ t+\frac{v}{n} < x $, we have
\begin{equation*}
|\gamma(v)|= n \left \vert \int_{0}^{\frac{1}{n}} \left(  \int_x^{t+\frac{v}{n}} \left( f^{(N)}(s)-f^{(N)}(x)\right) \frac{ \left( \left(t+\frac{v}{n} \right)-s\right)^{N-1}}{(N-1)!}ds\right)dt \right \vert
\end{equation*}
\begin{align*}
&\leq n\int_{0}^{\frac{1}{n}} \left(  \int^x_{t+\frac{v}{n}} \left\vert f^{(N)}(s)-f^{(N)}(x)\right \vert \frac{ \left(s- \left(t+\frac{v}{n} \right)\right)^{N-1}}{(N-1)!}ds\right)dt\\
&\leq n \int_{0}^{\frac{1}{n}} \omega \left( f^{(N)}, |t|+\left \vert \frac{v}{n}-x \right \vert \right)\left(  \int^x_{t+\frac{v}{n}} \frac{ \left(s- \left(t+\frac{v}{n} \right)\right)^{N-1}}{(N-1)!}ds\right)dt\\
&\leq \omega \left( f^{(N)}, \frac{1}{n}+\frac{1}{n^{\alpha}} \right) n  \int_{0}^{\frac{1}{n}} \frac{ \left(x- \left(t+\frac{v}{n} \right)\right)^{N}}{N!}dt\\
& \leq \omega \left( f^{(N)}, \frac{1}{n}+\frac{1}{n^{\alpha}} \right) n  \int_{0}^{\frac{1}{n}} \frac{ \left(\frac{1}{n}+\frac{1}{n^{\alpha}} \right)^{N}}{N!}dt \\
&=\omega \left( f^{(N)}, \frac{1}{n}+\frac{1}{n^{\alpha}} \right)    \frac{ \left(\frac{1}{n}+\frac{1}{n^{\alpha}} \right)^{N}}{N!}.
\end{align*}
\end{itemize}
By considering the above inequalities, we find that
\begin{equation*}
|\gamma(v)|\leq \omega \left( f^{(N)}, \frac{1}{n}+\frac{1}{n^{\alpha}} \right)    \frac{ \left(\frac{1}{n}+\frac{1}{n^{\alpha}} \right)^{N}}{N!}.
\end{equation*}
Then we conclude that 
\begin{equation*}
\left \vert n \int^{\infty}_{\substack{-\infty \\ \left \vert \frac{v}{n}-x \right \vert<\frac{1}{n^{\alpha}}}}\left( \int_{0}^{\frac{1}{n}} \left(  \int_x^{t+\frac{v}{n}} \left( f^{(N)}(s)-f^{(N)}(x)\right)\frac{ \left(t+\frac{v}{n}-s \right)^{N-1}}{(N-1)!}ds\right)dt\right) \Psi(nx-v)dv \right \vert
\end{equation*}
\begin{equation*}
\leq \omega \left( f^{(N)}, \frac{1}{n}+\frac{1}{n^{\alpha}} \right)    \frac{ \left(\frac{1}{n}+\frac{1}{n^{\alpha}} \right)^{N}}{N!}.
\end{equation*}
\begin{itemize}
\item[Case 2:]In this case,  we can write that
\end{itemize}
\begin{equation*}
\left \vert  \int_{\left \vert \frac{v}{n}-x \right \vert\geq\frac{1}{n^{\alpha}} } n \left( \int_{0}^{\frac{1}{n}} \left(  \int_x^{t+\frac{v}{n}} \left( f^{(N)}(s)-f^{(N)}(x)\right)\frac{ \left(t+\frac{v}{n}-s \right)^{N-1}}{(N-1)!}ds\right)dt\right) \Psi(nx-v)dv \right \vert
\end{equation*}
\begin{equation*}
\leq  \int_{\left \vert \frac{v}{n}-x \right \vert\geq\frac{1}{n^{\alpha}} } \left \vert n \left(  \int_{0}^{\frac{1}{n}} \left(  \int_x^{t+\frac{v}{n}} \left( f^{(N)}(s)-f^{(N)}(x)\right)\frac{ \left(t+\frac{v}{n}-s \right)^{N-1}}{(N-1)!}ds\right)d t\right)  \right\vert \Psi(nx-v)dv =:\mathtt{\Gamma} 
\end{equation*}
and
\begin{equation*}
|\gamma(v)|\leq n  \int_{0}^{\frac{1}{n}}\left \vert    \int_x^{t+\frac{v}{n}} \left( f^{(N)}(s)-f^{(N)}(x)\right) \frac{ \left(t+\frac{v}{n}-s \right)^{N-1}}{(N-1)!}ds \right \vert dt.
\end{equation*}
Now, if $t+\frac{v}{n}\geq x$, then
\begin{align*}
|\gamma(v)|\leq & 2 \left \Vert f^{(N)} \right \Vert_{\infty} n \int_0^\frac{1}{n} \frac{ \left(t+\frac{v}{n}-x \right)^{N}}{N!}dt\\
\leq& 2 \left \Vert f^{(N)} \right \Vert_{\infty} n \int_0^\frac{1}{n} \frac{ \left( |t|+\left \vert \frac{v}{n}-x \right \vert \right)^{N}}{N!}dt\\
\leq & 2 \frac{\left \Vert f^{(N)} \right \Vert_{\infty}}{N!} \left( \frac{1}{n}+\left \vert \frac{v}{n}-x \right \vert \right)^{N},
\end{align*}
and if  $t+\frac{v}{n} < x$ then $|\gamma(v)|\leq  2 \left \Vert f^{(N)} \right \Vert_{\infty} \frac{\left( \frac{1}{n}+\left \vert \frac{v}{n}-x \right \vert \right)^{N}}{N!}.$
By considering the above inequalities we conclude that 
\begin{equation*}
|\gamma(v)|\leq 2 \frac{\left \Vert f^{(N)} \right \Vert_{\infty}}{N!}  \left( \frac{1}{n}+\left \vert \frac{v}{n}-x \right \vert \right)^{N}
\end{equation*}
and by using Theorem \ref{t1} 
\begin{align*}
 \mathtt{\Gamma}&\leq \left( \int_{\left \vert \frac{v}{n}-x \right \vert\geq\frac{1}{n^{\alpha}} }  \left( \frac{1}{n}+\left \vert \frac{v}{n}-x \right \vert \right)^{N} \Psi(nx-v)dv \right)  \frac{ 2\left \Vert f^{(N)} \right \Vert_{\infty}}{N!} \\ 
& \leq \frac{ 2^N \left \Vert f^{(N)} \right \Vert_{\infty}}{N!} \int_{|nx-v|\geq n^{1-\alpha}} \left( \frac{1}{n^N}+ \frac{\left \vert nx-v\right \vert^N}{n^N}\right)\Psi(nx-v)dv\\
&\leq  \frac{2^N \left \Vert f^{(N)} \right \Vert_{\infty}}{n^N N!} \int_{|nx-v|\geq n^{1-\alpha}} \left( 1+ \left \vert nx-v\right \vert^N \right)\Psi(|nx-v|)dv\\
 &\leq \frac{ 2^N\left \Vert f^{(N)} \right \Vert_{\infty}}{n^N N!} \int_F \frac{1}{2}\left( q+\frac{1}{q} \right) \beta e^{-\beta \left(|nx-v|-1\right)}(1+|nx-v|)^N dv, ~F=\{v \in \mathbb{R}: |nx-v|\geq n^{1-\alpha} \}\\
 & \leq \frac{ 2^N\left \Vert f^{(N)} \right \Vert_{\infty}}{ n^N N!} \frac{1}{2}\left( q+\frac{1}{q} \right) \beta  \int_F e^{-\beta  \left(|nx-v|-1\right)}(1+|nx-v|)^N dv \\
 &= \frac{ 2^N\left \Vert f^{(N)} \right \Vert_{\infty}}{ n^N N!} \frac{1}{2}\left( q+\frac{1}{q} \right) \beta  2 \int_{n^{1-\alpha}}^{\infty} e^{-\beta  (x-1)} \left( 1+x^N \right)dx\\
 &= \frac{ 2^N\left \Vert f^{(N)} \right \Vert_{\infty}}{ n^N N!} \left( q+\frac{1}{q} \right)e^{\beta } \beta   \int_{n^{1-\alpha}}^{\infty} e^{-\beta  x} \left( 1+x^N \right)dx\\
 &= \frac{ 2^N\left \Vert f^{(N)} \right \Vert_{\infty}}{ n^N N!}\left( q+\frac{1}{q} \right)e^{ \beta } \beta   \left[  \int_{n^{1-\alpha}}^{\infty} e^{-\beta  x} dx+  \int_{n^{1-\alpha}}^{\infty} e^{-\beta  x} x^N dx \right]\\
 &= \frac{ 2^N\left \Vert f^{(N)} \right \Vert_{\infty}}{ n^N N!} \left( q+\frac{1}{q} \right)e^{\beta } \beta   \left[ \frac{ -e^{-\beta  x}}{ \beta } \bigg \vert_{n^{1-\alpha}}^{\infty}+  \int_{n^{1-\alpha}}^{\infty} e^{-\beta  x} x^N dx \right]\\
 &=\frac{ 2^N\left \Vert f^{(N)} \right \Vert_{\infty}}{ n^N N!} \left( q+\frac{1}{q} \right)e^{\beta } \beta   \left[ \frac{ e^{-\beta  (n^{1-\alpha})}}{ \beta }+  \int_{n^{1-\alpha}}^{\infty} e^{-\beta  x} x^N dx \right]=\mathtt{M}\\
\end{align*}
and
\begin{align*}
\int_{n^{1-\alpha}}^{\infty} e^{-\beta  x} x^N dx&=\frac{1}{\beta ^{N+1}} \int_{n^{1-\alpha}}^{\infty}  e^{-\beta  x} (\beta  x)^N (\beta  dx)\\
& =\frac{1}{\beta ^{N+1}} \int_{\beta  n^{1-\alpha}}^{\infty}  e^{-t} t^N dt\\
&  \leq \frac{1}{\beta ^{N+1}} \int_{\beta  n^{1-\alpha}}^{\infty}  e^{-t} e^{\frac{t}{2}} 2^N N!  dt\\
&=\frac{ 2^N N!}{\beta ^{N+1} } \int_{\beta  n^{1-\alpha}}^{\infty}  e^{-\frac{t}{2}} dt\\
&= \frac{ -2^{N+1} N!}{\beta ^{N+1} }  e^{-\frac{t}{2}} \bigg \vert_{\beta  n^{1-\alpha}}^{\infty}\\
&= \frac{ 2^{N+1} N!}{\beta ^{N+1}}  e^{-\frac{\beta  n^{1-\alpha}}{2}}.
\end{align*}
Then we obtain 
\begin{equation*}
\mathtt{M}\leq \frac{ 2^N\left \Vert f^{(N)} \right \Vert_{\infty}}{n^N N!} \left( q+\frac{1}{q} \right)e^{\beta } \beta \left[\frac{ e^{-\beta  (n^{1-\alpha})}}{ \beta }+ \frac{ 2^{N+1} N!}{\beta ^{N+1}}  e^{-\frac{\beta  n^{1-\alpha}}{2}} \right],
\end{equation*}
and since 
\begin{equation*}
\beta  n^{1-\alpha} >\frac{\beta  n^{1-\alpha}}{2}\Rightarrow -\beta  n^{1-\alpha}< \frac{-\beta  n^{1-\alpha}}{2}\Rightarrow  e^{-\beta  (n^{1-\alpha})} < e^{-\frac{\beta  n^{1-\alpha}}{2}}
\end{equation*}
we can easily write that 
\begin{align*}
\mathtt{M}&\leq\frac{ 2^N\left \Vert f^{(N)} \right \Vert_{\infty}}{n^N N!} \left( q+\frac{1}{q} \right)e^{\beta } \beta \left[\frac{ e^{-\beta  (n^{1-\alpha})}}{ \beta }+ \frac{ 2^{N+1} N!}{\beta ^{N+1}}  e^{-\frac{\beta  n^{1-\alpha}}{2}} \right]\\
& =\frac{ 2^N\left \Vert f^{(N)} \right \Vert_{\infty}}{n^N N!} \left( q+\frac{1}{q} \right)e^{\beta } e^{-\frac{\beta  n^{1-\alpha}}{2}} \left[1 +  \frac{ 2^{N+1} N!}{\beta ^{N}}  \right].
\end{align*}
Hence, we find 
\begin{equation*}
\left \vert  \int_{\left \vert \frac{v}{n}-x \right \vert\geq\frac{1}{n^{\alpha}} } n \left( \int_{0}^{\frac{1}{n}} \left(  \int_x^{t+\frac{v}{n}} \left( f^{(N)}(s)-f^{(N)}(x)\right)\frac{ \left(t+\frac{v}{n}-s \right)^{N-1}}{(N-1)!}ds\right)dt\right) \Psi(nx-v)dv \right \vert
\end{equation*}
\begin{equation*}
\leq \frac{ 2^N\left \Vert f^{(N)} \right \Vert_{\infty}}{n^N N!} \left( q+\frac{1}{q} \right)e^{\beta } e^{-\frac{\beta  n^{1-\alpha}}{2}} \left[1 +  \frac{ 2^{N+1} N!}{\beta ^{N}}  \right]\rightarrow 0,~ n\rightarrow \infty.
\end{equation*}
Also
\begin{align*}
\left \vert R_n (x) \right \vert &\leq \omega \left(f^{(N)}, \frac{1}{n}+\frac{1}{n^{\alpha}} \right) \frac{\left( \frac{1}{n}+\frac{1}{n^{\alpha}}\right)^N}{N!}\\
& + \frac{ 2^N\left \Vert f^{(N)} \right \Vert_{\infty}}{n^N N!} \left( q+\frac{1}{q} \right)e^{\beta } e^{-\frac{\beta  n^{1-\alpha}}{2}} \left[1 +  \frac{ 2^{N+1} N!}{\beta ^{N}}  \right]\rightarrow 0,~n\rightarrow \infty.
\end{align*}
So we notice for $k=1, 2, \cdots, N$ that 
\begin{align*}
\left \vert \mathtt{A}^*_n \left((\cdot -x)^k \right)\right \vert &= \left \vert  n \int_{-\infty}^{\infty} \left( \int_{0}^{\frac{1}{n}} \left( t+ \frac{v}{n}-x \right)^k dt \right) \Psi(nx-v)dv\right \vert\\
&\leq   n \int_{-\infty}^{\infty} \left(  \int_{0}^{\frac{1}{n}} \left \vert t+ \frac{v}{n}-x \right \vert^k dt \right) \Psi(nx-v)dv\\
& \leq n \int_{-\infty}^{\infty} \left( \int_{0}^{\frac{1}{n}} \left( | t|+ \left \vert \frac{v}{n}-x \right \vert \right)^k dt \right) \Psi(nx-v)dv\\
&\leq  \int_{-\infty}^{\infty}  \left(  \frac{1}{n}+ \left \vert \frac{v}{n}-x \right \vert \right)^k  \Psi(nx-v)dv\\
& =\frac{1}{n^k} \left[  \int_{-\infty}^{\infty}  \left( 1+|nx-v|\right)^k \Psi(nx-v)dv\right]\\
& \leq \frac{2^{k-1}}{n^k} \left[1+  \int_{-\infty}^{\infty}   |nx-v|^k \Psi(nx-v)dv\right]\\
& \leq \frac{2^{k-1}}{n^k} \left[1+  \int_{-\infty}^{\infty}   |h|^k  \Psi(h)dh\right]\\
& \leq \frac{2^{k-1}}{n^k} \left[1+  \frac{1-e^{-\beta }}{(1+e^{-\beta })}\frac{1}{k+1} +\left( q+\frac{1}{q} \right)\frac{ e^{\beta} k!}{\beta^k} \right]\rightarrow 0, ~as ~n\rightarrow \infty
\end{align*}
and we complete the proof.
\end{proof}
Our next result deals with  activated Quadrature operators.
\begin{theorem}
\label{12}
If $0<\alpha<1$, $n \in \mathbb{N}:n^{1-\alpha}>2$, $x \in \mathbb{R}$, $f \in C^{N}(\mathbb{R})$, $N \in \mathbb{N}$ with $f^{(N)} \in C_B(\mathbb{R})$. Then the followings hold: 
\begin{itemize}
\item[(i)]
\begin{align*}
\left|\overline{\mathtt{B}_n} (f)(x)-f(x) - \sum_{k=1}^N \frac{f^{(k)}(x)}{k!} \left(\overline{ \mathtt{A}_n} \left( (\cdot -x)^k \right) \right)(x)\right| 
\end{align*}
\begin{align*}
&\leq  \omega \left(f^{(N)}, \frac{1}{n}+\frac{1}{n^{\alpha}}\right) \frac{\left(  \frac{1}{n}+\frac{1}{n^{\alpha}}  \right)^N}{N!}\\
&+\frac{ 2^N\left \Vert f^{(N)} \right \Vert_{\infty}}{n^N N!} \left( q+\frac{1}{q} \right)e^{\beta } e^{-\frac{\beta  n^{1-\alpha}}{2}} \left[1 +  \frac{ 2^{N+1} N!}{\beta ^{N}}\right] \rightarrow 0,~as~ n\rightarrow \infty,
\end{align*}
\item[(ii)]
if $ f^{(k)}(x)=0$, $k=1, 2, \cdots, N$ then 
\begin{align*}
\vert \overline{\mathtt{B}_n}(f)(x)-f(x)\vert &\leq  \omega \left(f, \frac{1}{n}+\frac{1}{n^{\alpha}}\right) \frac{\left(   \frac{1}{n}+\frac{1}{n^{\alpha}} \right)^N}{N!}\\
&+\frac{ 2^N\left \Vert f^{(N)} \right \Vert_{\infty}}{n^N N!} \left( q+\frac{1}{q} \right)e^{\beta } e^{-\frac{\beta  n^{1-\alpha}}{2}} \left[1 +  \frac{ 2^{N+1} N!}{\beta ^{N}}\right],
\end{align*}
\item[(iii)]
\begin{align*}
\vert\overline{\mathtt{B}_n}(f)(x)-f(x)\vert& \leq \sum_{k=1}^N \frac{\left \vert f^{(k)}(x)\right \vert}{k!} \frac{2^{k-1}}{n^k} \left[1+ \frac{1-e^{-\beta }}{(1+e^{-\beta })}\frac{1}{k+1} +\left( q+\frac{1}{q} \right)\frac{ e^{\beta} k!}{\beta^k} \right]\\
&+  \omega \left(f^{(N)}, \frac{1}{n}+\frac{1}{n^{\alpha}}\right) \frac{\left(   \frac{1}{n}+\frac{1}{n^{\alpha}} \right)^N}{N!}\\
&+ \frac{ 2^N\left \Vert f^{(N)} \right \Vert_{\infty}}{n^N N!} \left( q+\frac{1}{q} \right)e^{\beta } e^{-\frac{\beta  n^{1-\alpha}}{2}} \left[1 +  \frac{ 2^{N+1} N!}{\beta ^{N}}\right],
\end{align*}
\item[(iv)]
\begin{align*}
\Vert\overline{\mathtt{B}_n}(f)-f \Vert_{\infty}& \leq \sum_{k=1}^N \frac{\left\Vert f^{(k)} \right \Vert_{\infty}}{k!} \frac{2^{k-1}}{n^k} \left[1+ \frac{1-e^{-\beta }}{(1+e^{-\beta })}\frac{1}{k+1} +\left( q+\frac{1}{q} \right)\frac{ e^{\beta} k!}{\beta^k}  \right]  \\
&+ \omega \left(f^{(N)}, \frac{1}{n}+\frac{1}{n^{\alpha}}\right) \frac{\left(   \frac{1}{n}+\frac{1}{n^{\alpha}} \right)^N}{N!}\\
&+\frac{ 2^N\left \Vert f^{(N)} \right \Vert_{\infty}}{n^N N!} \left( q+\frac{1}{q} \right)e^{\beta } e^{-\frac{\beta  n^{1-\alpha}}{2}} \left[1 +  \frac{ 2^{N+1} N!}{\beta ^{N}}\right].
\end{align*}
\end{itemize}
\end{theorem}
\begin{proof}
Since it is already known that 
\begin{align*}
f \left( \frac{v}{n}+\frac{s}{nr} \right)&=\sum_{k=0}^N \frac{f^{(k)}(x)}{k!}\left( \frac{v}{n}+\frac{s}{nr} -x \right)^k\\
&+ \int_{x}^{\frac{v}{n}+\frac{s}{nr}} \left( f^{(N)}(t)- f^{(N)}(x) \right) \frac{\left( \frac{v}{n}+\frac{s}{nr} -t \right)^{N-1}}{(N-1)!}dt
\end{align*}
and 
\begin{align*}
\sum_{s=1}^r w_s f \left( \frac{v}{n}+\frac{s}{nr} \right)&=\sum_{k=0}^N \frac{f^{(k)}(x)}{k!} \sum_{s=1}^r w_s\left( \frac{v}{n}+\frac{s}{nr} -x \right)^k\\
&+ \sum_{s=1}^r w_s \int_{x}^{\frac{v}{n}+\frac{s}{nr}} \left( f^{(N)}(t)- f^{(N)}(x) \right) \frac{\left( \frac{v}{n}+\frac{s}{nr} -t \right)^{N-1}}{(N-1)!}dt,
\end{align*}
we can write that  
\begin{align*}
\overline{\mathtt{B}_n}(f)(x)&= \int_{- \infty}^{\infty} \left( \sum_{s=1}^r w_s f \left( \frac{v}{n}+\frac{s}{nr} \right) \right)\Psi(nx-v)dv\\
&= \sum_{k=0}^N \frac{f^{(k)}(x)}{k!} \left(    \int_{- \infty}^{\infty} \sum_{s=1}^r w_s \left( \frac{v}{n}+\frac{s}{nr} -x \right)^k\           \right)\Psi(nx-v)dv\\
&+ \int_{- \infty}^{\infty} \left( \sum_{s=1}^r w_s \int_{x}^{\frac{v}{n}+\frac{s}{nr}} \left( f^{(N)}(t)- f^{(N)}(x) \right) \frac{\left( \frac{v}{n}+\frac{s}{nr} -t \right)^{N-1}}{(N-1)!} dt\right)\Psi(nx-v)dv,
\end{align*}
and 
\begin{align*}
\overline{\mathtt{B}_n}(f)(x)-f(x)& = \sum_{k=1}^N \frac{ f^{(k)} (x)}{k!} \left( \overline{\mathtt{A}_n} \left((\cdot -x)^k \right)(x)\right) +R_n(x), 
\end{align*}
where
\begin{equation*}
R_n(x):=\int_{- \infty}^{\infty} \left( \sum_{s=1}^r w_s \int_{x}^{\frac{v}{n}+\frac{s}{nr}}  \left( f^{(N)}(t)- f^{(N)}(x) \right) \frac{\left( \frac{v}{n}+\frac{s}{nr} -t \right)^{N-1}}{(N-1)!} dt\right)\Psi(nx-v)dv
\end{equation*}
and  calling 
\begin{equation*}
\gamma(v):=\sum_{s=1}^r w_s \int_{x}^{\frac{v}{n}+\frac{s}{nr}} \left( f^{(N)}(t)- f^{(N)}(x) \right) \frac{\left( \frac{v}{n}+\frac{s}{nr} -t \right)^{N-1}}{(N-1)!}dt.
\end{equation*}
Now we have two cases: \begin{itemize}
\item[1.]  $ \left \vert \frac{v}{n}-x\right \vert< \frac{1}{n^{\alpha}},$
\item[2.]  $ \left \vert \frac{v}{n}-x\right \vert \geq \frac{1}{n^{\alpha}}.$
\end{itemize}
Let us start with Case 1. 
\begin{itemize}
\item[Case 1 (i):] When  $\frac{v}{n}+\frac{s}{nr}\geq x$, we have 
\begin{align*}
|\gamma(v)|&\leq\sum_{s=1}^r w_s \omega \left( f^{(N)}, \left \vert \frac{v}{n}+\frac{s}{nr} -x\right \vert \right) \frac{\left( \frac{v}{n}+\frac{s}{nr} -x \right)^{N}}{N!}\\
&\leq \omega \left(f^{(N)}, \frac{1}{n}+\frac{1}{n^{\alpha}}\right) \frac{\left(   \frac{1}{n}+\frac{1}{n^{\alpha}} \right)^N}{N!}.
\end{align*}
\item[Case 1 (ii):] When $\frac{v}{n}+\frac{s}{nr}<x$, we have
\begin{align*}
|\gamma(v)|&= \left \vert \sum_{s=1}^r w_s \int^{x}_{\frac{v}{n}+\frac{s}{nr}} \left( f^{(N)}(t)- f^{(N)}(x) \right) \frac{ \left( t-\left( \frac{v}{n}+\frac{s}{nr} \right) \right)^{N-1}}{(N-1)!}dt \right \vert \\
&\leq \sum_{s=1}^r w_s \int^{x}_{\frac{v}{n}+\frac{s}{nr}}  \left \vert f^{(N)}(t)- f^{(N)}(x) \right \vert \frac{ \left( t-\left( \frac{v}{n}+\frac{s}{nr} \right) \right)^{N-1}}{(N-1)!}dt\\
&\leq\sum_{s=1}^r w_s \omega \left( f^{(N)}, \left \vert x-\left( \frac{v}{n}+\frac{s}{nr} \right) \right \vert \right) \frac{\left( x-\left( \frac{v}{n}+\frac{s}{nr} \right)\right)^{N}}{N!}\\
&\leq \omega \left(f^{(N)}, \frac{1}{n}+\frac{1}{n^{\alpha}}\right) \frac{\left(   \frac{1}{n}+\frac{1}{n^{\alpha}} \right)^N}{N!}.
\end{align*}
\end{itemize}
By considering the above inequalities,  we find that 
\begin{equation*}
|\gamma(v)|\leq \omega \left(f^{(N)}, \frac{1}{n}+\frac{1}{n^{\alpha}}\right) \frac{\left(   \frac{1}{n}+\frac{1}{n^{\alpha}} \right)^N}{N!}
\end{equation*}
and
\begin{equation*}
\left \vert \int_{\left \vert \frac{v}{n}-x \right \vert< \frac{1}{n^{\alpha}}} \left( \sum_{s=1}^r w_s \int_{x}^{\frac{v}{n}+\frac{s}{nr}} \left( f^{(N)}(t)- f^{(N)}(x) \right) \frac{\left( \frac{v}{n}+\frac{s}{nr} -t \right)^{N-1}}{(N-1)!}dt \right)\Psi(nx-v)dv \right \vert
\end{equation*} 
\begin{equation*}
 \leq \int_{\left \vert \frac{v}{n}-x \right \vert< \frac{1}{n^{\alpha}}} \Psi(nx-v) |\gamma (v)|dv \leq \omega \left(f^{(N)}, \frac{1}{n}+\frac{1}{n^{\alpha}}\right) \frac{\left(   \frac{1}{n}+\frac{1}{n^{\alpha}} \right)^N}{N!}.
\end{equation*}
\begin{itemize}
\item[Case 2:] In this case, we can write that
\end{itemize}
\begin{equation*}
\left \vert \int_{\left \vert \frac{v}{n}-x \right \vert \geq \frac{1}{n^{\alpha}}} \Psi(nx-v)  \left( \sum_{s=1}^r w_s \int_{x}^{\frac{v}{n}+\frac{s}{nr}} \left( f^{(N)}(t)- f^{(N)}(x) \right) \frac{\left( \frac{v}{n}+\frac{s}{nr} -t \right)^{N-1}}{(N-1)!}dt \right)dv \right \vert
\end{equation*}
\begin{equation*}
 \leq \int_{\left \vert \frac{v}{n}-x \right \vert\geq \frac{1}{n^{\alpha}}} \Psi(nx-v) |\gamma (v)|dv =:\xi.
\end{equation*}
First let us consider the Case  $\frac{v}{n}+\frac{s}{nr} \geq x$. Then we have 
\begin{equation*}
 |\gamma(v)|\leq \frac{ 2 \left \Vert f^{(N)} \right \Vert_{\infty}}{N!} \sum_{s=1}^r w_s \left( \frac{v}{n}+\frac{s}{nr}-x \right)^N.
\end{equation*}
Now let us assume that  $\frac{v}{n}+\frac{s}{nr} < x$. Then we have 
\begin{align*}
|\gamma(v)|&= \left \vert \sum_{s=1}^r w_s \int^{x}_{\frac{v}{n}+\frac{s}{nr}} \left( f^{(N)}(t)- f^{(N)}(x) \right) \frac{ \left( t-\left( \frac{v}{n}+\frac{s}{nr} \right) \right)^{N-1}}{(N-1)!}dt \right \vert \\
&\leq \sum_{s=1}^r w_s \int^{x}_{\frac{v}{n}+\frac{s}{nr}}  \left \vert f^{(N)}(t)- f^{(N)}(x) \right \vert \frac{ \left( t-\left( \frac{v}{n}+\frac{s}{nr} \right) \right)^{N-1}}{(N-1)!}dt\\
&\leq  \frac{ 2 \left \Vert f^{(N)} \right \Vert_{\infty}}{N!} \sum_{s=1}^r w_s \left( x-\left(\frac{v}{n}+\frac{s}{nr} \right) \right)^N
\end{align*}
and consequently we can write for all cases 
\begin{equation*}
 |\gamma(v)|\leq \frac{ 2 \left \Vert f^{(N)} \right \Vert_{\infty}}{N!}\left( \left \vert x-\frac{v}{n}\right \vert+\frac{1}{n} \right)^N.
\end{equation*}
Similarly, as in the earlier theorem, we have that 
\begin{equation*}
 \left \vert  \int_{\left \vert \frac{v}{n}-x \right \vert \geq \frac{1}{n^{\alpha}}} \Psi(nx-v) \gamma (v)dv \right \vert\leq \xi
\end{equation*}
and
\begin{equation*}
\xi \leq\frac{ 2^N\left \Vert f^{(N)} \right \Vert_{\infty}}{n^N N!} \left( q+\frac{1}{q} \right)e^{\beta } e^{-\frac{\beta  n^{1-\alpha}}{2}} \left[1 +  \frac{ 2^{N+1} N!}{\beta ^{N}}\right]\rightarrow 0,~as~ n\rightarrow \infty.
\end{equation*}
We also see for $k=1, 2, \cdots, N$ that
\begin{align*}
\left\vert \overline{\mathtt{B}_n} \left((\cdot -x)^k \right)(x)\right \vert&=\left\vert   \int_{- \infty}^{\infty} \left( \sum_{s=1}^r w_s\left( \frac{v}{n}+\frac{s}{nr} -x \right)^k \right)\Psi(nx-v)dv \right \vert\\
&\leq \int_{- \infty}^{\infty} \left( \sum_{s=1}^r w_s \left\vert \frac{v}{n}+\frac{s}{nr} -x  \right \vert^k \right)\Psi(nx-v)dv\\
& \leq  \int_{- \infty}^{\infty} \left( \sum_{s=1}^r w_s \left(\left\vert  \frac{v}{n} -x  \right \vert +\frac{s}{nr} \right)^k \right)\Psi(nx-v)dv\\
&  \leq  \int_{- \infty}^{\infty} \left( \frac{1}{n}+\left\vert  \frac{v}{n} -x  \right \vert  \right)^k \Psi(nx-v)dv\\
& =\frac{1}{n^k} \int_{- \infty}^{\infty} \left( 1+\left\vert nx -v\right \vert  \right)^k \Psi(nx-v)dv\\
& \leq \frac{2^{k-1}}{n^k} \left[1+ \frac{1-e^{-\beta }}{(1+e^{-\beta })}\frac{1}{k+1} +\left( q+\frac{1}{q} \right)\frac{ e^{\beta} k!}{\beta^k} \right]\rightarrow 0, ~as ~n\rightarrow \infty.
\end{align*}
Hence we complete the proof.
\end{proof}
\begin{theorem}
If $0<\alpha<1$, $n \in \mathbb{N}:n^{1-\alpha}>2$, $x \in \mathbb{R}$, $f^{(k)} \in C^{N}(\mathbb{R})$, $N \in \mathbb{N}$, $k=0, 1, \cdots, s \in \mathbb{N}$ with $f^{(N+k)} \in C_B(\mathbb{R})$. Then 
\begin{itemize}
\item[(i)]
\begin{align*}
\left| (\mathtt{B}_n (f))^{(k)} (x)-f^{(k)}(x) - \sum_{m=1}^N \frac{f^{(k+m)}(x)}{m!} \left( \mathtt{A}_n\left( (\cdot -x)^{m} \right) \right)(x) \right|
\end{align*}
\begin{equation*}
\leq \frac{\omega \left( f^{(N+k)}, \frac{1}{n^{\alpha}} \right)}{n^{\alpha N}N!}+ \frac{2^{N+2}  \Vert f^{(N+k)} \Vert_{\infty}e^{\beta } \left( q+\frac{1}{q} \right)}{n^N \beta ^N } e^{ \frac{-\beta  n^{1-\alpha}}{2}},
\end{equation*}
\item[(ii)]
\begin{align*}
\left|( \mathtt{B}^*_n (f))^{(k)}(x)-f^{(k)}(x) - \sum_{m=1}^N \frac{f^{(k+m)}(x)}{m!} \left(\mathtt{A}^*_n\left( (\cdot -x)^m\right) \right)(x)\right|
\end{align*}
\begin{align*}
&\leq  \omega \left(f^{(N+k)}, \frac{1}{n}+\frac{1}{n^{\alpha}}\right) \frac{\left(  \frac{1}{n}+\frac{1}{n^{\alpha}}  \right)^N}{N!}\\
&+ \frac{ 2^N\left \Vert f^{(N+k)} \right \Vert_{\infty}}{n^N N!} \left( q+\frac{1}{q} \right)e^{\beta } e^{-\frac{\beta  n^{1-\alpha}}{2}} \left[1 +  \frac{ 2^{N+1} N!}{\beta ^{N}}\right],
\end{align*}
\item[(iii)]
\begin{align*}
\left| (\overline{\mathtt{B}_n} (f))^{(k)}(x)-f^{(k)}(x) - \sum_{m=1}^N \frac{f^{(k+m)}(x)}{m!} \left( \overline{ \mathtt{A}_n} \left( (\cdot -x)^{m} \right) \right)(x)\right| 
\end{align*}
\begin{align*}
&\leq  \omega \left(f^{(N+k)}, \frac{1}{n}+\frac{1}{n^{\alpha}}\right) \frac{\left(  \frac{1}{n}+\frac{1}{n^{\alpha}}  \right)^N}{N!}\\
&+\frac{ 2^N\left \Vert f^{(N+k)} \right \Vert_{\infty}}{ n^N N!}\left( q+\frac{1}{q} \right)e^{\beta } e^{-\frac{\beta  n^{1-\alpha}}{2}} \left[1 +  \frac{ 2^{N+1} N!}{\beta ^{N}}\right].
\end{align*}
\end{itemize}
\end{theorem}
\begin{proof}
By using Theorems \ref{10}, \ref{11}, \ref{12}, we immediately obtain the proof. 
\end{proof}
\section{Iterated Version of $\mathtt{B}_n$, $\mathtt{B}^*_n$ $\overline{\mathtt{B}_n}$ }
In this section, we consider the iterated versions of our operators  and examine  approximation properties  of them under the light of \cite{anast 2022},  \cite{G.A.2024a},  \cite{G.A.2024}, \cite{G.A.2024c}, \cite{G.A.2024d}, \cite{G.A.2024e}, \cite{anast 2025}.
\begin{remark}(About Iterated Convolution)\\
Notice that 
\begin{equation*}
\mathtt{B}_n(f)(x)=\int_{-\infty}^{ \infty} f \left(x-\frac{h}{n}\right) \Psi(h)dh, ~for~f \in C_{B}(\mathbb{R})
\end{equation*}
and let $x_{k}\rightarrow x$, as $k\rightarrow \infty$, and 
\begin{equation*}
\mathtt{B}_n(f)(x_{k})-\mathtt{B}_n(f)(x)=\int_{-\infty}^{ \infty} 
\left[f \left(x_k-\frac{h}{n}\right)-f \left(x-\frac{h}{n}\right)\right] \Psi(h)dh.
\end{equation*}
We have that 
\begin{equation*}
f \left(x_k-\frac{h}{n}\right) \Psi(h)\rightarrow f \left(x-\frac{h}{n}\right)\Psi(h), ~for~every~ h \in \mathbb{R}, ~ k \rightarrow \infty.
\end{equation*}
\end{remark}
Furthermore
\begin{equation*}
|\mathtt{B}_n(f)(x_{k})-\mathtt{B}_n(f)(x)|
\leq \int_{-\infty}^{ \infty} 
\left \vert f \left(x_k-\frac{h}{n}\right)-f \left(x-\frac{h}{n}\right) \right \vert \Psi(h) dh\rightarrow 0, ~k \rightarrow \infty,
\end{equation*}
by Dominated Convergence Theorem, since 
\begin{equation*}
\left \vert f \left(x_k-\frac{h}{n}\right) \right \vert \Psi(h)dh \leq \Vert f \Vert_{\infty}  \Psi(h)
\end{equation*}
and $\Vert f \Vert_{\infty}  \Psi(h) $ is integrable over $\left( -\infty, \infty \right)$, for every $  h \in \left( -\infty, \infty \right)$.
Thus $\mathtt{B}_n(f) \in C_{B}(\mathbb{R})$.\\
 Also we have that
\begin{equation*}
|\mathtt{B}_n(f)(x)|\leq \Vert f \Vert _{\infty} \int_{-\infty}^{ \infty} \Psi(h)dh =\Vert f \Vert_{\infty}
\end{equation*}
i.e.,
\begin{equation*}
\Vert \mathtt{B}_n(f) \Vert _{\infty} \leq \Vert f \Vert _{\infty}
\end{equation*}
 which means that $B_n$ is bounded and linear for $n \in \mathbb{N}.$
\begin{remark} 
Let $r \in \mathbb{N}$. Since 
\begin{align*}
\mathtt{B}_n^r f-f&=(\mathtt{B}_n^r f-\mathtt{B}_n^{r-1} f)+(\mathtt{B}_n^{r-1} f+\mathtt{B}_n^{r-2} f)+(\mathtt{B}_n^{r-2} f+\mathtt{B}_n^{r-3} f)\\
&+\ldots + (\mathtt{B}_n^2 f-\mathtt{B}_n f)+(\mathtt{B}_n^r f-f).
\end{align*}
\end{remark} 
We have that  $\Vert \mathtt{B}_n^r f-f \Vert _{\infty} \leq r \Vert \mathtt{B}_n f-f\Vert _{\infty} $ and 
\begin{align*}
\mathtt{B}_{k_r}(\mathtt{B}_{k_{r-1}}(\ldots \mathtt{B}_{k_2}(\mathtt{B}_{k_1} f)))-f&=\ldots=\mathtt{B}_{k_r}(\mathtt{B}_{k_{r-1}}(\ldots \mathtt{B}_{k_2}))(\mathtt{B}_{k_1}-f)\\
&+\mathtt{B}_{k_r}(\mathtt{B}_{k_{r-1}}(\ldots \mathtt{B}_{k_3}))(\mathtt{B}_{k_2}-f)\\
&+\mathtt{B}_{k_r}(\mathtt{B}_{k_{r-1}}(\ldots \mathtt{B}_{k_4}))(\mathtt{B}_{k_3}-f)\mathtt{B}_{k_r}\\
&+\ldots+(\mathtt{B}_{k_{r-1}} f-f)+\mathtt{B}_{k_r}f-f
\end{align*}
where $k_1, k_2, \ldots, k_r\in \mathbb{N}:k_1\leq k_2\leq \ldots \leq k_r $ and
\begin{equation*}
\Vert \mathtt{B}_{k_r}(\mathtt{B}_{k_{r-1}}(\ldots \mathtt{B}_{k_2}(\mathtt{B}_{k_1} f)))-f\Vert_{\infty}\leq \sum_{m=1}^r\Vert \mathtt{B}_{k_m}f-f\Vert_{\infty}
\end{equation*}
as in  \cite{G.A.2024}. The similar results can be obtained for $\mathtt{B}^*_n$ and $ \overline{\mathtt{B}_n}$ using the same technique. 
\begin{remark}
Notice that 
\begin{equation*}
\mathtt{B}^*_n(f)(x)=n \int_{-\infty}^{ \infty} \left( \int^{\frac{1}{n}}_0 f \left( t+\left(x-\frac{h}{n}\right)\right)dt\right) \Psi(h) dh, ~f\in C_{B}(\mathbb{R}),
\end{equation*}
and let $x_{k}\rightarrow x$, as $k\rightarrow \infty$,
\end{remark}
and
\begin{equation*}
\left\vert \mathtt{B}^*_n (f)(x_k)-\mathtt{B}^*_n (f)(x)\right\vert 
\end{equation*}
\begin{align*}
\leq& n \int_{-\infty}^{ \infty} \left( \int^{\frac{1}{n}}_0 \left \vert f \left( t+\left(x_k-\frac{h}{n}\right)\right) - f \left( t+\left(x-\frac{h}{n}\right)\right) \right \vert dt\right) \Psi(h)dh,
\end{align*}
  by  Bounded Convergence Theorem, we get that:
\begin{equation*}
x_k\rightarrow x\rightsquigarrow t+\left(x_k-\frac{h}{n}\right)\rightarrow t+\left(x-\frac{h}{n}\right)
\end{equation*}
and 
\begin{equation*}
f \left( t+\left(x_k-\frac{h}{n}\right)\right)\rightarrow f \left( t+\left(x-\frac{h}{n}\right)\right), 
\end{equation*}
\begin{equation*}
\left \vert f \left( t+\left(x_k-\frac{h}{n}\right)\right)\right \vert \leq \Vert f \Vert_{\infty}
\end{equation*}
and $\left[ 0, \frac{1}{n} \right ] $ is finite. Hence
\begin{equation*}
n \int^{\frac{1}{n}}_0 \left \vert f \left( t+\left(x_k-\frac{h}{n}\right)\right) - f \left( t+\left(x-\frac{h}{n}\right)\right) \right \vert dt \rightarrow0, ~as~ k\rightarrow \infty.
\end{equation*}
Therefore it holds  that
\begin{equation*}
 n\int^{\frac{1}{n}}_0 f \left( t+\left(x_k-\frac{h}{n}\right)\right) dt\rightarrow n \int^{\frac{1}{n}}_0 f \left( t+\left(x-\frac{h}{n}\right)\right) dt \rightarrow0, ~as~ k\rightarrow \infty
\end{equation*}
and
\begin{equation*}
 \left( n\int^{\frac{1}{n}}_0 f \left( t+\left(x_k-\frac{h}{n}\right)\right) dt \right)\Psi(h)\rightarrow \left( n \int^{\frac{1}{n}}_0 f \left( t+\left(x-\frac{h}{n}\right)\right) dt \right)\Psi(h), 
\end{equation*}
as $k\rightarrow \infty, ~for~every~ h \in \left( -\infty, \infty \right)$.
Also, we get 
\begin{equation*}
 \left \vert \left( n\int^{\frac{1}{n}}_0 f \left( t+\left(x_k-\frac{h}{n}\right)\right) dt \right)\Psi(h) \right \vert \leq \Vert f \Vert_{\infty}\Psi(h).
\end{equation*}
Again by Dominated Convergence Theorem,
\begin{equation*}
\mathtt{B}^*_n (f)(x_k)\rightarrow \mathtt{B}^*_n (f)(x),~as~ k\rightarrow \infty.
\end{equation*}
Thus  $ \mathtt{B}^*_n (f)(x)$ is bounded and continuous in $x \in \left( -\infty, \infty \right)$ and the iterated facts hold for $\mathtt{B}^*_n$ as in the $ \mathtt{B}_n (f)$ case, all the same. \\
\begin{remark}
Next we observe that: $f\in C_{B}(\mathbb{R})$, and
\begin{equation}
 \overline{\mathtt{B}_n}(f)(x):=\int_{-\infty}^{\infty} \left( \sum_{s=1}^r w_s f \left( \left(x-\frac{h}{n}\right)+\frac{s}{nr}\right)\right) \Psi(h)dh.
\end{equation}
Let $x_{k}\rightarrow x$, as $k\rightarrow \infty$.
\end{remark}
 Then
\begin{align*}
\left \vert \overline{\mathtt{B}_n}(f)(x_k)-\overline{\mathtt{B}_n}(f)(x)\right \vert &=\left\vert \int_{-\infty}^{ \infty} \left(  \sum_{s=1}^r w_s \left( f \left( \left(x_k-\frac{h}{n}\right)+\frac{s}{nr}\right)- f \left( \left(x-\frac{h}{n}\right)+\frac{s}{nr}\right)\right) \right) \Psi(h)dh \right\vert\\
&\leq  \int_{-\infty}^{ \infty}  \left\vert \sum_{s=1}^r w_s \left( f \left( \left(x_k-\frac{h}{n}\right)+\frac{s}{nr}\right)- f \left( \left(x-\frac{h}{n}\right)+\frac{s}{nr} \right) \right)\right\vert \Psi(h)dh \rightarrow 0, ~as~ k\rightarrow \infty.
\end{align*}
The last is from Dominated Convergence Theorem:
\begin{equation*}
\left(x_k-\frac{h}{n}\right)+\frac{s}{nr}\rightarrow \left(x-\frac{h}{n}\right)+\frac{s}{nr}
\end{equation*}
and
\begin{equation*}
\sum_{s=1}^r w_s f \left( \left(x_k-\frac{h}{n}\right)+\frac{s}{nr}\right)\rightarrow \sum_{s=1}^r w_s f \left( \left(x-\frac{h}{n}\right)+\frac{s}{nr} \right)
\end{equation*}
and 
\begin{equation*}
\left(\sum_{s=1}^r w_s f \left( \left(x_k-\frac{h}{n}\right)+\frac{s}{nr}\right)\right)\Psi(h)\rightarrow \left(\sum_{s=1}^r w_s \left( \left(x-\frac{h}{n}\right)+\frac{s}{nr} \right)\right)\Psi(h),
\end{equation*}
as $k\rightarrow \infty, \forall h \in  \left( -\infty, \infty \right)$.
Furthermore, we have 
\begin{equation*}
 \left \vert \sum_{s=1}^r w_s f \left( \left(x_k-\frac{h}{n}\right)+\frac{s}{nr}\right) \right \vert \Psi(h)  \leq \Vert f \Vert_{\infty}\Psi(h)
\end{equation*}
which the last one function is integrable over $ \left( -\infty, \infty \right)$.\\
Therefore 
\begin{equation*}
\overline{\mathtt{B}_n}(f)(x_k)\rightarrow \overline{\mathtt{B}_n}(f)(x),~as~ k\rightarrow \infty.
\end{equation*}
Thus $ \overline{\mathtt{B}_n} (f)(x)$ is  bounded and continuous in $x \in \left( -\infty, \infty \right)$ and the iterated facts holds all the same.
\begin{theorem}
If  $0<\alpha<1$, $n \in \mathbb{N}:n^{1-\alpha}>2$, $r \in \mathbb{N}$, $f \in C_B(\mathbb{R})$, then we have the followings:
\begin{itemize}
\item[(i)] \begin{equation*}
\left \Vert \mathtt{B}_n^r (f) -f \right \Vert_{\infty} \leq r \left \Vert \mathtt{B}_n f -f \right \Vert_{\infty} \leq r \left[ \omega\left(f,\frac{1}{n^{\alpha}}\right)+ \frac{2\left( q+\frac{1}{q} \right)  \Vert f \Vert_{\infty}}{e^{\beta (n^{1-\alpha}-1)}} \right],
\end{equation*}
\item[(ii)] \begin{equation*}
\left \Vert \mathtt{B}_n^{*r} (f) -f \right \Vert_{\infty} \leq r \left \Vert \mathtt{B}^*_n f -f \right \Vert_{\infty} \leq r \left[ \omega\left(f,\frac{1}{n}+\frac{1}{n^{\alpha}}\right)+ \frac{2\left( q+\frac{1}{q} \right)  \Vert f \Vert_{\infty}}{e^{\beta (n^{1-\alpha}-1)}} \right],
\end{equation*}
\item[(iii)] \begin{equation*}
\left \Vert\overline{\mathtt{B}_n}^{r} (f) -f \right \Vert_{\infty} \leq r \left \Vert\overline{\mathtt{B}_n} f -f \right \Vert_{\infty} \leq r \left[ \omega\left(f,\frac{1}{n}+\frac{1}{n^{\alpha}}\right)+ \frac{2\left( q+\frac{1}{q} \right)  \Vert f \Vert_{\infty}}{e^{\beta (n^{1-\alpha}-1)}} \right].
\end{equation*}
\end{itemize}
The rate of convergence of $\mathtt{B}_n^r $, $\mathtt{B}_n^{*r} $, $\overline{\mathtt{B}_n}^{r}$ to identity  is not worse than the rate of convergence of  $\mathtt{B}_n $, $\mathtt{B}_n^{*} $, $\overline{\mathtt{B}_n}$.
\end{theorem}
\begin{proof}
The proof follows from Theorems  \ref{theorem3}, \ref{theorem4}, \ref{theorem5}.
\end{proof}
\begin{theorem}
If $0<\alpha<1$, $n \in \mathbb{N}; k_1, k_2, \cdots, k_r \in \mathbb{N}:k_1\leq k_2 \leq  \cdots \leq k_r $, $ k_p^{1-\alpha}>2$, $p=1, 2, \cdots, r;$  $f \in C_B(\mathbb{R})$,  then we have the followings:
\begin{itemize}
\item[(i)] \begin{align*}
\Vert \mathtt{B}_{k_r}(\mathtt{B}_{k_{r-1}}(\ldots \mathtt{B}_{k_2}(\mathtt{B}_{k_1} f)))-f\Vert_{\infty}&\leq \sum_{p=1}^r\Vert \mathtt{B}_{k_p}f-f\Vert_{\infty}\\
& \leq \sum_{p=1}^r  \left[ \omega\left(f,\frac{1}{k^{\alpha}_p}\right)+ \frac{2\left( q+\frac{1}{q} \right)  \Vert f \Vert_{\infty}}{e^{\beta (k_p^{1-\alpha}-1)}} \right] \\
&\leq r \left[ \omega\left(f,\frac{1}{k^{\alpha}_1}\right)+ \frac{2\left( q+\frac{1}{q} \right)  \Vert f \Vert_{\infty}}{e^{\beta (k_1^{1-\alpha}-1)}} \right],
\end{align*}
\item[(ii)] \begin{align*}
\Vert \mathtt{B}^*_{k_r}(\mathtt{B}^*_{k_{r-1}}(\ldots \mathtt{B}^*_{k_2}(\mathtt{B}^*_{k_1} f)))-f\Vert_{\infty}&\leq \sum_{p=1}^r\Vert \mathtt{B}^*_{k_p}f-f\Vert_{\infty}\\
& \leq \sum_{p=1}^r \left[ \omega\left(f,\frac{1}{k_p}+\frac{1}{k_p^{\alpha}}\right)+ \frac{2\left( q+\frac{1}{q} \right)  \Vert f \Vert_{\infty}}{e^{\beta (k_p^{1-\alpha}-1)}} \right] \\
&\leq r \left[ \omega\left(f,\frac{1}{k_1}+\frac{1}{k_1^{\alpha}}\right)+ \frac{2\left( q+\frac{1}{q} \right)  \Vert f \Vert_{\infty}}{e^{\beta (k_1^{1-\alpha}-1)}} \right],
\end{align*}
\item[(iii)] \begin{align*}
\Vert\overline{\mathtt{B}}_{k_r}(\mathtt{B}^*_{k_{r-1}}(\ldots \overline{\mathtt{B}}_{k_2}(\mathtt{B}^*_{k_1} f)))-f\Vert_{\infty}&\leq \sum_{p=1}^r\Vert \overline{\mathtt{B}}_{k_p}f-f\Vert_{\infty}\\
& \leq \sum_{p=1}^r \left[ \omega\left(f,\frac{1}{k_p}+\frac{1}{k_p^{\alpha}}\right)+ \frac{2\left( q+\frac{1}{q} \right)  \Vert f \Vert_{\infty}}{e^{\beta (k_p^{1-\alpha}-1)}} \right] \\
&\leq r \left[ \omega\left(f,\frac{1}{k_1}+\frac{1}{k_1^{\alpha}}\right)+ \frac{2\left( q+\frac{1}{q} \right)  \Vert f \Vert_{\infty}}{e^{ \beta (k_1^{1-\alpha}-1)}} \right].
\end{align*}
\end{itemize}
\end{theorem}
\begin{proof}
The proof follows from Theorems  \ref{theorem3}, \ref{theorem4}, \ref{theorem5}.
\end{proof}


\begin{thebibliography}{99}
\bibitem{Arif} Arif, A., Yurdakadim, T. (2025). Approximation results on neural network operators of convolution type. \emph{https://arxiv.org/pdf/2503.22301}.
\bibitem{Ismailaslan} Aslan, İ. (2025).  Approximation by Max-Min Neural Network Operators. \emph{Numerical Functional Analysis and Optimization}. 46(4–5): 374–393.
\bibitem{ALADAg} Aladağ, Ç.H., Eğrioğlu, E., Kadılar, C. (2009).  Forecasting nonlinear time series with a hybrid methodology.\emph{ Applied Mathematics Letters}. 22: 1467-1470.
\bibitem{anast 1985} Anastassiou, G.A. (1985). A $K$-attainable inequality related to the convergence of positive linear operators. \emph{ J. Approx. Theory} 44: 380-383.
\bibitem{anast 1993} Anastassiou, G.A. (1993). \emph{Moments in probability and approximation theory}.  Pitman Research Notes in Mathematics series/Longman Scientific \& Technical, Essex, New York, UK, USA.
\bibitem{anast 2001}  Anastassiou, G.A. (2001).  \emph{Quantitative Approximations}. New
York, Boca Raton: Chapman and Hall/CRC.
\bibitem{anast 2022} Anastassiou, G.A. (2022). q-Deformed and parametrized half hyperbolic tangent based Banach space valued multivariate multi layer neural network approximations. \emph{Dep. of Math. Sci., Univ. of Memphis, USA}.
\emph{J. Comput. Anal. Appl.} 31(4): 520–534.
\bibitem{Anastassiou20232} Anastassiou, G.A.  (2023).  \emph{ Intelligent Computations: Parametrized. Deformed and General Neural Networks}. New York, Heidelberg: Springer.
\bibitem{G.A.2024a} Anastassiou, G.A. (2024). $q$-Deformed and $\lambda$-parametrized $A$-generalized logistic function induced
Banach space valued multivariate multi layer neural network approximations. \emph{ Stud. Univ. Babeş-Bolyai Math.} 69(3): 587-612.
\bibitem{G.A.2024}  Anastassiou, G.A. (2024). Approximation by parametrized logistic activated convolution type operators.
\emph{ Rev. Real Acad. Cienc. Exactas Fis. Nat. Ser. A-mat.}: 118-138.
\bibitem{G.A.2024c} Anastassiou, G.A. (2024). Multivariate approximation by parametrized logistic activated multidimensional convolution type operators. \emph{Annals of Com. in Math.} 7(2): 128-159.
\bibitem{G.A.2024d} Anastassiou, G.A. (2024). Approximation by symmetrized and perturbed hyperbolic tangent activated convolution type operators. \emph{Mathematics 2024}. 12, 3302. 
\bibitem{G.A.2024e} Anastassiou, G.A. (2024). Multivariate Approximation Using Symmetrized and Perturbed Hyperbolic
Tangent-Activated Multidimensional Convolution-Type Operators. \emph{Dep. of Math. Sci., Univ. of Memphis, USA}. 13, 779.
\bibitem{anast 2025}Anastassiou, G.A. (2025). \emph{ Trigonometric and Hyperbolic Generated Approximation Theory}. World Scientific, Singapore, New York.
\bibitem{G.A.2025b} Anastassiou, G.A. (2025). Complete approximation by symmetrized and perturbed generalized logistic activated multidimensional convolutions in positive linear framework. \emph{Rev. Real Acad. Cienc. Exactas Fis. Nat. Ser. A-mat.} 119:96.
\bibitem{ARMSTRONG} Armsrtrong , J.S., Fildes, R. (2006).  Making progress in forecasting. \emph{International Journal of Forecasting}.  22(3): 433-441.
\bibitem{amatod} Amatod, P., Fariand, M., Montagnab, G., Morellia, M.J., Nicrosina , O., Treccanib, M. (2004). Pricing financial derivatives with neural networks. \emph{Physica A: Statistical Mechanics and its Applications}. 338(1-2): 160-165.
\bibitem{BISHOP} Bishop, C.(1995). \emph{Neural Networks For Pattern Recognition} Birmingham, UK: Clarendon Press.
\bibitem{Bojanic} Bojanic, R., Shisha, O. (1973). On the precision of uniform approximation of continuous functions by certain linear operators of convolution type. \emph{J. Approx. Theory 8}. 101–113.
\bibitem{Cardaliaguet} Cardaliaguet, P., Euvrard, G. (1992).  Approximation of a function and its derivative with a neural network. \emph{ Neural Networks}. 5(2): 207-220.
\bibitem{CHUNG} Chung, K.C.,  Tan, S.S.  Holdsworth, D.K. (2008). Insolvency prediction model using multivariate discriminant analysis and artificial neural network for the finance industry in New Zealand \emph{International Journal of Business and Management}. 39(1): 19-28.
\bibitem{cIFTER} Çifter, A., Özün, A. (2008). Modelling long-term memory effect in stock prices. \emph{ Studies in Economics and Finance}. 25(1): 38-48.
\bibitem{FAALSIDE}  Niranjan, M.,  Faalside, F. (1990). Neural networks and radial basis functions in classifying static speech patterns. \emph{ Computer Speech and Language}. 4(3): 275-289.
\bibitem{Jung} Jung, H.S., Sakai, R. (2014). Local saturation of a positive linear convolution operator. \emph{J. Inequal. Appl. 2014}. 329, 16.
\bibitem{Moldovan2} Moldovan, G. (1973).  Discrete convolutions and linear positive operators. \emph{Ann. Univ. Sci. Budapest. Eötvös Sect. Math. 15(1972)}. 31–44.
\bibitem{Moldovan} Moldovan, G. (1974). Discrete convolutions in connection with functions of several variables and positive linear operators. \emph{(Romanian) Studia Univ. Babeş-Bolyai Ser. Math.-Mech.} 19(1): 51–57.
\bibitem{Kamzolov} Kamzolov, A.I. (1970). The order of approximation of functions of class $Z_2(E_n)$ by positive linear convolution
operators. \emph{(Russian) Mat. Zametki 7}. 723–732. 
\bibitem{PARKER} Parker, D.B. (1985). Learning-logic, Technical Report No. 47, Center for Computational Research in Economics and Management Science:  MIT Press. Cambridge. MA.
\bibitem{rozenbaltt}  Rosenblatt, F. (1957). The Perceptron: A Perceiving and Recognizing Automaton. Report 85-460-1.  Buffalo. New York: Cornell Aeronautical Laboratory.
\bibitem{Swetits} Swetits, J.J., Wood, B. (1982): Local $L_p$ -saturation of positive linear convolution operators.\emph{ J. Approx. Theory}. 34(4): 348–360.
\end{thebibliography}
\end{document}